\documentclass[11pt, leqno]{amsart}

\usepackage{amsmath}
\usepackage{amssymb,amsfonts}
\usepackage{cancel}
\usepackage{comment}
\usepackage{enumerate}
\usepackage{enumitem}
\usepackage{geometry}
\usepackage{hyperref}
\usepackage{a4wide}
\usepackage{mathrsfs}
\usepackage[dvipsnames]{xcolor}
\usepackage[all,arc]{xy}

\newtheorem{thm}{Theorem}[section]
\newtheorem{cor}[thm]{Corollary}
\newtheorem{prop}[thm]{Proposition}
\newtheorem{lem}[thm]{Lemma}

\theoremstyle{definition}
\newtheorem{defn}[thm]{Definition}

\theoremstyle{remark}
\newtheorem{rem}[thm]{Remark}

\makeatletter
\let\c@equation\c@thm
\makeatother
\numberwithin{equation}{section}

\title{Classification of quadratically pinched translators in higher codimension}

\thanks{The authors are partially supported by INdAM-GNSAGA}



\author[Debora Impera]{Debora Impera}
\address[Debora Impera]{Dipartimento di Scienze Matematiche ``Giuseppe Luigi Lagrange", Politecnico di Torino, Corso Duca degli Abruzzi, 24, Torino, Italy, I-10129}
\email{debora.impera@polito.it}

\author[Michele Rimoldi]{Michele Rimoldi}
\address[Michele Rimoldi]{Dipartimento di Scienze Matematiche ``Giuseppe Luigi Lagrange", Politecnico di Torino, Corso Duca degli Abruzzi, 24, Torino, Italy, I-10129}
\email{michele.rimoldi@polito.it}

\author[Francesco Ruatta]{Francesco Ruatta}
\address[Francesco Ruatta]{Dipartimento di Scienze Matematiche ``Giuseppe Luigi Lagrange", Politecnico di Torino, Corso Duca degli Abruzzi, 24, Torino, Italy, I-10129}
\email{francesco.ruatta@polito.it}

\hypersetup{
    colorlinks=true,
    linkcolor=RawSienna,
    citecolor=RawSienna,
    filecolor=magenta,      
    urlcolor=cyan,
    pdftitle={Classification of quadratically pinched translators in higher codimension},
    pdfpagemode=FullScreen,
}

\allowdisplaybreaks

\begin{document}

\begin{abstract}\label{abstract}
We study complete translating solitons for the mean curvature flow in arbitrary codimension under quadratic pinching of the second fundamental form. We prove a codimension reduction theorem at the Andrews-Baker threshold: if a complete $n$-dimensional connected translator isometrically immersed in $(\mathbb{R}^{n+p},\langle\cdot,\cdot\rangle)$ satisfies
\[
|B|^2\leq \left(\frac{4}{3n}-\varepsilon_{0}\right)|H|^2
\]
for some $\varepsilon_{0}>0$, then either it is an affine $n$-plane or $|H|>0$ everywhere and the translator is contained in an $(n+1)$-dimensional affine subspace. In particular, this improves in low dimensions the pinching range previously obtained from codimension reduction results for ancient mean curvature flows. Our proof is purely elliptic: viewing translators as weighted minimal submanifolds, we combine Simons-type identities for the drift Laplacian with refined gradient estimates and an Omori-Yau maximum principle. As a consequence, under the stronger pinching condition
\[
|B|^2\leq (c_n-\varepsilon_{0})|H|^2,\qquad
c_n=\min\left\{\frac{4}{3n},\frac{1}{n-2}\right\}
\]
for $n\geq3$, with $c_n=2/3$ when $n=2$, we obtain a rigidity theorem: the translator is either an affine plane or a bowl soliton contained in an $(n+1)$-dimensional affine subspace. No entropy assumption is required.
\end{abstract}

\maketitle


\section{Introduction}

Given $n\geq 2$ and $p\geq 1$, let $x_{0}:M^{n}\to\left(\mathbb R^{n+p},\langle\cdot,\cdot\rangle\right)$ be a smooth
isometric immersion of a connected $n$-dimensional Riemannian manifold without boundary as a
complete submanifold of the $(n+p)$-dimensional Euclidean space. A one-parameter family
$\{x(\cdot,t)\}_{t\in[t_0,T)}$ of smooth immersions, with images $M_t\doteq x(\cdot,t)(M)$, is a
solution of the \textit{mean curvature flow} (MCF) starting from $M_{t_0}$ if
\begin{equation}\label{eq:MCF}
    \begin{cases}
        \partial_t x(\cdot,t)=\mathbf H(\cdot,t),\\
        x(\cdot,t_0)=x_0,
    \end{cases}
\end{equation}
where $\mathbf H(\cdot,t)$ denotes the mean curvature vector of $M_t$ at $x(\cdot,t)$, defined as the
trace of the second fundamental form $B$ with respect to the induced metric $g$.

A solution of \eqref{eq:MCF} is called ancient if $t_0=-\infty$, while it is called eternal if it is defined for every time $t\in(-\infty,+\infty)$. The simplest non-static examples of eternal solutions are the translating solitons. They evolve by ambient translations, and they appear in the singularity theory of the flow as blow-up limits at type II singularities;
see, for example, \cite{Hamilton_HarnackEstimateMCF,HuiskenSinestrari_ConvexityEstimates}. Explicitly, given a non-zero vector $\vec v\in\mathbb R^{n+p}$, a smooth
isometric immersion $x:M^n\to\mathbb R^{n+p}$ is a \textit{translator} with velocity $\vec v$ if it
satisfies the system
\begin{equation}\label{eq:translator}
    \mathbf H=\vec v^{\perp},
\end{equation}
where $(\cdot)^{\perp}$ is the projection onto the normal bundle $NM$ of $M$. The scaling $x\to |\vec v|x$ maps a translator with velocity $\vec v$ to one with velocity $\vec v/|\vec v|$ and it leaves \eqref{eq:pinching_definition} introduced below invariant, so, without loss of generality, we will always assume that $|\vec v|=1$. The smooth family of immersions $x:M^n\times\mathbb R\to\mathbb R^{n+p}$ given by $x(\cdot,t)\doteq x_0(\cdot)+t\vec v$ is, up to tangential reparametrizations, a solution of \eqref{eq:MCF}.

Equivalently, as proved in Section \ref{sec:translators}, translators are exactly the $\rho$-minimal immersions for the weight $\rho(\cdot)=-\langle \vec v,\cdot\rangle$.
The associated operator is the drift Laplacian,
\begin{equation*}
    \Delta_\rho \varphi\doteq\Delta \varphi-\langle\nabla\rho,\nabla \varphi\rangle=\Delta \varphi+\langle \vec v^{\top},\nabla \varphi\rangle.
\end{equation*}

In codimension one, translators are well understood under convexity assumptions. Precisely, up to rigid motions, there exists a unique rotationally symmetric and strictly convex translator (see, e.g., \cite{AltschulerWu_TranslatingSurfaces}). For $n=1$ it is the grim reaper curve \cite{Mullins_GrainBoundaries}, while for $n\geq 2$ it is called the \emph{bowl soliton}. Wang \cite{Wang_ConvexSolutionsMCF} proved that, when $n=2$, every entire convex translator is the bowl soliton. More generally, Spruck and Xiao \cite{SpruckXiao} proved that every complete, two sided, mean convex translator in $\mathbb{R}^{3}$ is convex. In greater dimensions, Haslhofer \cite[Theorem 1.2]{Haslhofer_BowlSoliton} proved that every uniformly $2$-convex and $\alpha$-noncollapsed translator is the bowl soliton. The latter condition can be guaranteed, for complete, orientable and uniformly $2$-convex translators, by combining \cite[Theorem 1.1]{SpruckSun_ConvexityTranslatingSolitons}, Sacksteder's theorem \cite{Sacksteder_HypersurfacesNoNegSectionalCurvatures} and \cite[Theorem 1.1]{BourniLangfordLynch_CollapsingAndNoncollapsingAncientMCF}, as in the proof of Theorem \ref{thm:rigidity_result}. 

In higher codimension, a natural substitute for convexity is provided by the so-called
\textit{quadratic pinching} condition. A submanifold is said to be \emph{$c$-quadratically pinched}
if it satisfies the pointwise inequality
\begin{equation}\label{eq:pinching_definition}
    |B|^{2}\leq c |\mathbf H|^{2},
\end{equation}
and \emph{uniformly $c$-quadratically pinched} if $|B|^{2}\leq(c-\varepsilon_0)|\mathbf H|^{2}$ for
some $\varepsilon_0>0$. This condition plays a fundamental role mostly for two reasons. First, the pinching condition $|B|^2+a\leq c|\mathbf H|^2$, for some $a>0$, is preserved along the flow when $c\leq\frac{4}{3n}$, as shown by Andrews and Baker in \cite[Theorem 2]{AndrewsBaker_MCFPinchedSubmanifolds}. Second, in codimension $p=1$, convexity properties are reflected by suitable choices of the pinching constant $c$. Specifically, every uniformly $\frac{1}{n-k}$-quadratically pinched hypersurface (with $1 \leq k \leq n-1$) is uniformly $k$-convex; see \cite[Lemma 5.1]{HuiskenSinestrari_MCFConvexAncientSolutions}.

An expected phenomenon under quadratic pinching is \emph{codimension reduction}. For $n\geq 5$ and $p>1$, this was established by Naff \cite{Naff_Planarity} for blow-up limits of compact pinched flows satisfying $|B|^2 < c_n^{\mathrm{Naff}} |\mathbf{H}|^2$.
It was later extended by Lee, Naff, and Zhu \cite{LeeNaffZhu_Planarity} to uniformly pinched ancient solutions under the same pinching constant:
\begin{equation*}
    c_n^{\mathrm{Naff}}=\min\left\{\frac{3(n+1)}{2n(n+2)}, \frac{4}{3n}\right\}
    =\begin{cases}
    \dfrac{3(n+1)}{2n(n+2)}, & n\leq 7,\\[6pt]
    \dfrac{4}{3n}, & n\geq 8.
    \end{cases}
\end{equation*}

Since a uniformly $c_n^{\mathrm{Naff}}$-pinched translator generates an eternal solution of the flow with bounded curvature, \cite[Theorem 1.1]{LeeNaffZhu_Planarity} applied on $[0,T]$, with $T\to+\infty$, shows that the component of $B$ orthogonal to $\mathbf H$ vanishes unless $B\equiv0$. Codimension reduction follows from \cite[Proposition 2.5]{Naff_Planarity}. The constant $c_n^{\mathrm{Naff}}$ is, however, strictly smaller than the Andrews--Baker threshold $\frac{4}{3n}$ exactly when $n\leq 6$, with equality for $n\geq 7$. This gap has already been shown, in the self-shrinker setting, to be due to technical aspects of the approach in \cite{Naff_Planarity,LeeNaffZhu_Planarity}; see \cite{IRR_sshc_classification}. 
The goal of this paper is to prove that the same holds true for translators, that is, to prove codimension
reduction up to the uniform $\frac{4}{3n}$-pinching in \emph{every} dimension. Our
main result is the following.

\begin{thm}[Codimension reduction]\label{thm:main_codim_reduction}
    Given $n\geq 2$ and $p\geq 1$, let $x:M^{n}\to\mathbb R^{n+p}$ be a complete, connected translator
    without boundary, satisfying $\mathbf H=\vec v^{\perp}$ with $|\vec v|=1$. Furthermore, assume that
    \begin{equation}\label{eq:main_pinching}
        |B|^{2}\leq\left(\frac{4}{3n}-\varepsilon_0\right)|\mathbf H|^{2}
        \qquad\text{on }M,
    \end{equation}
    for some $\varepsilon_0>0$. Then one of the following holds:
    \begin{enumerate}
        \item $B\equiv 0$ and $x$ is an isometry onto an $n$-dimensional plane containing $\vec v$;
        \item $|\mathbf H|>0$ everywhere; the component $B_I$ of the second fundamental form orthogonal to the principal normal $\nu\doteq \mathbf H/|\mathbf H|$ vanishes, $\nu$ is parallel in the normal bundle and $x(M)$ is contained in the affine subspace $A\doteq x(q_0)+W\subset\mathbb{R}^{n+p}$, for some $q_0\in M$ and some $(n+1)$-dimensional linear subspace $W\subseteq \mathbb R^{n+p}$. Moreover, $\vec v\in W$, the normal space of $x$ in $A$ is $\mathbb R\nu$, and $x:M\to A$ is an orientable codimension one translator with the same velocity $\vec v$.
    \end{enumerate}
\end{thm}

    For $n\geq 7$ one has $c_n^{\mathrm{Naff}}=\frac{4}{3n}$, and the claim follows from Lee, Naff and Zhu argument, as explained above. Instead, when $n\in\{3,4\}$, one has that $\frac{4}{3n}\leq \frac{1}{n-1}$, and the result is a consequence of the compactness criterion \cite[Proposition 4.1]{LeeNaffZhu_Planarity} together with the fact that there exist no closed translators. We extend this criterion to the case $n=2$ in Proposition \ref{prop:translator_flatness_dim_234_intro}. Therefore, the genuinely new cases are $n\in\{2,5,6\}$.

In the following, we keep the framework of \cite{Naff_Planarity,LeeNaffZhu_Planarity}, together with their reaction and gradient estimates. The differences lie in the fact that the evolution equations of \cite{Naff_Planarity} are replaced by Simons identities for the drift Laplacian $\Delta_\rho$ and in the treatment of the gradient terms. Indeed, in \cite[Appendix B.2]{LeeNaffZhu_Planarity} they are shown to be nonnegative pointwise, and this is where the constant $\frac{3(n+1)}{2n(n+2)}$ comes from. We only bound them from below by a term that vanishes at almost-maximum points, and we close the argument with the Omori--Yau maximum principle for the drift Laplacian, following, therefore, a purely elliptic approach in the spirit of \cite{IRR_sshc_classification}. The resulting proof, given in Section \ref{sec:codim_reduction_n_geq_5}, holds for every $n\geq5$.

Theorem \ref{thm:main_codim_reduction} reduces the problem to the hypersurface case, yielding rigidity under uniform $2$-convexity. Specifically, combining \cite{HuiskenSinestrari_MCFConvexAncientSolutions,SpruckSun_ConvexityTranslatingSolitons,Sacksteder_HypersurfacesNoNegSectionalCurvatures,BourniLangfordLynch_CollapsingAndNoncollapsingAncientMCF,Haslhofer_BowlSoliton}, we prove:

\begin{thm}[Rigidity result]\label{thm:rigidity_result}
    Let $x:M^n\to\mathbb R^{n+p}$ be a uniformly $c_n$-pinched translator, where
    \begin{equation*}
        c_n\doteq\begin{cases}
            \frac{4}{3n}\qquad &2\leq n\leq 8,\\
            \frac{1}{n-2}&n\geq 9,
        \end{cases}
    \end{equation*}
    that is, $c_n=\min\left\{\frac{4}{3n},\frac{1}{n-2}\right\}$ for $n\geq3$. Suppose $x$ satisfies all the other hypotheses of Theorem \ref{thm:main_codim_reduction}. Then either $x$ is an isometry onto an $n$-dimensional plane containing $\vec{v}$, or $x$ is an embedding whose image is a bowl soliton contained in an $(n+1)$-dimensional affine subspace of $\mathbb R^{n+p}$. In particular, for $n\in\{2,3,4\}$ only the first case is possible.
\end{thm}

Theorem \ref{thm:rigidity_result} is the translator counterpart of \cite[Corollary 1.5]{LeeNaffZhu_Planarity}. It requires no entropy bound and, for $n\in\{2,5,6\}$, it holds under uniform $c_n$-pinching instead of uniform $c_n^{\mathrm{Naff}}$-pinching, which gives a weaker pinching assumption in these dimensions.

The paper is organized as follows. In Section \ref{sec:preliminaries} we fix the notation and we recall the basic equations for submanifolds. Section \ref{sec:translators} describes translators as $\rho$-minimal immersions and proves the Omori--Yau maximum principle for $\Delta_\rho$. In Section \ref{sec:pinched_submanifolds} we introduce useful quantities and identities on quadratically pinched submanifolds. Section \ref{section_simons_equations} contains the Simons-type identities for $|\mathbf H|^2$, $|B|^2$, $f$, $|B_\nu|^2$ and $|B_I|^2$, and Section \ref{sec:estimates} collects the algebraic and gradient estimates of \cite{Naff_Planarity}. Section \ref{sec:codim_reduction_n_geq_5} proves the codimension reduction for $n\geq5$. Section \ref{sec:codim_reduction_low} states and proves the aforementioned compactness criterion, and treats the codimension reduction in the $2\leq n\leq4$ case. Lastly, Theorems \ref{thm:main_codim_reduction} and \ref{thm:rigidity_result} are proved in Section \ref{sec:proof_main_theorems}. Appendix \ref{sec:algebraic_lemma_5_6} contains an algebraic inequality used in Section \ref{sec:codim_reduction_n_geq_5}.

\subsection*{Added note} As we were making the final revisions to this manuscript, Lynch \cite{Lynch26} independently extended Naff's parabolic planarity estimate in dimensions $5$ and $6$ to the full Andrews-Baker pinching range for compact mean curvature flows, using a new sharp Kato-type inequality. This closes the corresponding gap in the compact parabolic setting. The present work concerns instead complete translators without boundary, which are necessarily noncompact, and proves codimension reduction by a direct elliptic argument based on drift-Laplacian identities and the Omori-Yau maximum principle. Thus, although both results reach the same critical pinching range, they apply to genuinely different geometric settings and rely on different mechanisms.

\subsection*{AI Disclosure}
ChatGPT and Claude were used only for stylistic and editorial revisions of the manuscript, as well as for a final check of the mathematical arguments. All mathematical content was independently verified by the authors, who take full responsibility for the results and proofs presented in the paper.

\section{Notation and preliminaries on submanifolds}\label{sec:preliminaries}

In this section we recall the basic definitions and equations for submanifolds of the Euclidean space, and we fix the notation used throughout the paper.

Let $x:M^n\to (\mathbb R^{n+p},\langle\cdot,\cdot\rangle)$ be a smooth immersion of an $n$-dimensional manifold $M^n$ into the $(n+p)$-dimensional Euclidean space, and let $g$ be the induced metric on $M$ via the pullback. We denote by $TM$ and $NM$, respectively, the tangent and normal bundles of $M$ in $\mathbb R^{n+p}$, and by $\Gamma( \cdot )$ the space of smooth sections of a bundle.
The $(1,3)$- and $(0,4)$-curvature tensors of $(M,g)$ are defined with the following sign conventions
\begin{align*}
    R(X,Y)Z&\doteq \nabla_X(\nabla_YZ)-\nabla_Y(\nabla_XZ)-\nabla_{[X,Y]}Z,
    \\ R(X,Y,Z,W)&\doteq \langle R(X,Y)W,Z\rangle.
\end{align*}
For tensor fields $\mathcal F\in\Gamma((T^*M)^{\otimes k})$ we set $\nabla^2_{X,Y}\mathcal F\doteq\nabla_X(\nabla_Y\mathcal F)-\nabla_{\nabla_XY}\mathcal F$ and denote by $\Delta\doteq\operatorname{tr}_g\nabla^2$ the rough Laplacian, which on functions is the Laplace--Beltrami operator.

We use the following ranges of indices:
\begin{equation*}
    1\leq i,j,k,\ldots\leq n,\qquad n+1\leq\alpha,\beta,\gamma,\ldots\leq n+p,
\end{equation*}
and we let $\{e_i\}_{i=1}^n$ and $\{\nu_\alpha\}_{\alpha=n+1}^{n+p}$ be local orthonormal frames of $TM$ and $NM$. On the open set $\{|\mathbf H|>0\}$ we always choose the adapted frame in which $\nu_{n+1}$ is the principal normal $\mathbf H/|\mathbf H|$, denoted simply by $\nu$. Sums over Latin indices range from $1$ to $n$, while sums over Greek indices from $n+1$ to $n+p$, unless otherwise stated.

The Ricci tensor of $(M,g)$ is obtained by tracing the $(0,4)$-curvature tensor as
\begin{equation}\label{eq:ricci_trace_convention}
    \operatorname{Ric}(X,Y)\doteq\sum_i R(X,e_i,Y,e_i).
\end{equation}

\subsection{Connections}\label{subsec:connections}

The computations of the following sections require carefully distinguishing among covariant derivatives acting on different bundles, in order to avoid overloading the symbol $\nabla$. For this purpose, we write $\overline\nabla$ for the flat connection of the ambient space $\mathbb R^{n+p}$, acting on vector fields along the immersion (i.e. on sections of $x^*T\mathbb R^{n+p}=TM\oplus NM$), and $\nabla$ for the Levi-Civita connection of $(M,g)$, acting on (tensor powers of) $TM$ and $T^*M$. We then set
\begin{equation}\label{eq:normal_connection_definition}
    \nabla^\perp_X\xi\doteq(\overline\nabla_X\xi)^\perp,\qquad X\in\Gamma(TM),\ \xi\in\Gamma(NM).
\end{equation}
Consistently with the above convention, the curvature of the normal connection is defined by
\begin{equation*}
    R^\perp(X,Y)\xi\doteq\nabla^\perp_X(\nabla^\perp_Y\xi)-\nabla^\perp_Y(\nabla^\perp_X\xi
    )-\nabla^\perp_{[X,Y]}\xi,\quad R^\perp(X,Y,\eta,\xi)\doteq\langle R^\perp(X,Y)\xi,\eta\rangle .
\end{equation*}

Throughout, $\overline\nabla_i\doteq\overline\nabla_{e_i}$, $\nabla_i\doteq\nabla_{e_i}$ and
$\nabla^\perp_i\doteq\nabla^\perp_{e_i}$. Furthermore we will use the symbol $\nabla^\perp$ not only to denote the normal connection \eqref{eq:normal_connection_definition}, but also for the connection it induces, together with $\nabla$, on bundles of the form $(T^*M)^{\otimes k}\otimes NM$. Equivalently, for $\mathcal T\in\Gamma((T^*M)^{\otimes k}\otimes NM)$ and $X,Y_1,\ldots,Y_k\in\Gamma(TM)$,
\begin{equation*}
    (\nabla^\perp_X\mathcal T)(Y_1,\ldots,Y_k)=\nabla^\perp_X(\mathcal T(Y_1,\ldots,Y_k))-\sum_{a=1}^k\mathcal T(Y_1,\ldots,\nabla_XY_a,\ldots,Y_k),
\end{equation*}
where on the right-hand side $\nabla^\perp$ is the connection \eqref{eq:normal_connection_definition} on $NM$. Since both $\nabla$ and $\nabla^\perp$ are compatible with the metric, so is the induced connection. That is $X\langle \mathcal T,\mathcal T'\rangle=\langle\nabla^\perp_X\mathcal T,\mathcal T'\rangle+\langle \mathcal T,\nabla^\perp_X\mathcal T'\rangle$ for $\mathcal T,\mathcal T'\in\Gamma((T^*M)^{\otimes k}\otimes NM)$. Here, $\langle\cdot,\cdot\rangle$ denotes the metric induced by $g$ and by the ambient metric on
$(T^*M)^{\otimes k}\otimes NM$. In particular, in the local orthonormal frames $\{e_i\}$ and $\{\nu_\alpha\}$,
the squared norm of the second fundamental form (to be defined in Subsection \ref{subsec:sff}) and of its gradient are
\begin{equation*}
    |B|^2=\sum_{i,j}|h_{ij}|^2=\sum_{i,j,\alpha}(h^\alpha_{ij})^2,
    \qquad
    \bigl|\nabla^\perp B\bigr|^2=\sum_{i,j,k,\alpha}(h^\alpha_{ijk})^2,
\end{equation*}
and analogously for $\mathbf H$, $\mathring B$, $B_\nu$, $B_I$ and their derivatives.

We set $\Delta^\perp\doteq \operatorname{tr}_g(\nabla^\perp)^2$ for the associated rough Laplacian. The symbols $\nabla^\perp B$, $\nabla^\perp\mathbf H$, $\Delta^\perp B$ and $\Delta^\perp\mathbf H$ are understood in this sense.

\subsection{The second fundamental form}\label{subsec:sff}

The \emph{second fundamental form} $B$ of $M$ is the normal part of the ambient covariant derivative,
\begin{equation*}
    B:\Gamma(TM)\times \Gamma(TM)\longrightarrow \Gamma(NM),\qquad B(X,Y)\doteq (\overline\nabla_XY)^\perp=\overline\nabla_XY-\nabla_XY,
\end{equation*}
so that $B\in\Gamma(\mathrm{Sym}^2T^*M\otimes NM)$.
In the chosen orthonormal frames, $B$ is expressed as
\begin{equation*}
    B=\sum_\alpha h^\alpha\otimes\nu_\alpha,\qquad h^\alpha_{ij}=\langle \overline\nabla_i e_j, \nu_\alpha \rangle,
\end{equation*}
where the $h^\alpha$'s are symmetric $2$-tensors. We will write
\begin{equation}\label{eq:sff_vector_valued_components}
    h_{ij}\doteq B(e_i,e_j)=\sum_\alpha h^\alpha_{ij}\nu_\alpha
\end{equation}
for the corresponding $NM$-valued quantity.

The \emph{mean curvature vector} $\mathbf H$ of $M$ is the trace of the second fundamental form $B$.
Under the assumption that $\mathbf H$ is nowhere vanishing, the principal normal $\nu=\mathbf H/|\mathbf H|$ is a globally well-defined unit
section of $NM$. This induces an orthogonal splitting of the normal bundle: $NM = \mathbb{R}\nu \oplus \mathcal{I}$, where $\mathcal{I} \doteq \nu^\perp \subset NM$.
Setting
\begin{equation*}
    \sigma^\beta_k\doteq\langle\nabla^\perp_k\nu,\nu_\beta\rangle,
\end{equation*}
we have
\begin{equation*}
    \sigma^{n+1}_k=0,\qquad \nabla^\perp_k\nu=\sum_{\beta\geq n+2}\sigma^\beta_k\nu_\beta,\qquad \bigl|\nabla^\perp\nu\bigr|^2=\sum_k\sum_{\beta\geq n+2}(\sigma^\beta_k)^2.
\end{equation*}
The tensor
$\nabla^\perp\nu\in\Gamma(T^*M\otimes\mathcal I)$ measures the failure of $\nu$ to be parallel in the
normal bundle. Differentiating $\mathbf H=|\mathbf H|\nu$ along $X\in\Gamma(TM)$ yields
\begin{equation}\label{eq:nabla_perp_H_split}
    \nabla^\perp_X\mathbf H=(\nabla_X|\mathbf H|) \nu+|\mathbf H| \nabla^\perp_X\nu.
\end{equation}

Turning back to the second fundamental form, if $|\mathbf H|>0$ then, in the adapted frame, one has $\operatorname{tr} h^{n+1}=|\mathbf H|$ and $\operatorname{tr} h^\alpha=0$ for $\alpha\geq n+2$. The \emph{trace-free second fundamental form} is $\mathring B\doteq B-\frac1ng\otimes\mathbf H$, with components
\begin{equation}\label{eq:tracefree_sff_components}
    \mathring h^{n+1}\doteq h^{n+1}-\frac{|\mathbf H|}{n}g,\qquad \mathring h^\alpha=h^\alpha\quad(\alpha\geq n+2).
\end{equation}
Setting
\begin{equation*}
    B_\nu\doteq h^{n+1}\otimes\nu,\qquad \mathring B_\nu\doteq \mathring h^{n+1}\otimes\nu,\qquad B_I\doteq \sum_{\alpha\geq n+2} h^\alpha\otimes\nu_\alpha,
\end{equation*}
we obtain the orthogonal decompositions $B=B_\nu+B_I$ and $\mathring B=\mathring B_\nu+ B_I$, along with the identities
\begin{equation}\label{eq:norm_decompositions}
    |B_I|^2=|B|^2-|B_\nu|^2,\qquad |\mathring B_\nu|^2=|B_\nu|^2-\frac{|\mathbf H|^2}{n}.
\end{equation}

We denote by
\begin{equation}\label{eq:first_covariant_derivative_component_sff}
    h^\alpha_{ijk}\doteq\bigl\langle(\nabla^\perp_k B)(e_i,e_j),\nu_\alpha\bigr\rangle,
    \qquad
    h^\alpha_{ijkl}\doteq\bigl\langle\bigl((\nabla^\perp)^2_{l,k}B\bigr)(e_i,e_j),\nu_\alpha\bigr\rangle
\end{equation}
the components of $\nabla^\perp B$ and $(\nabla^\perp)^2B$. Consistently with \eqref{eq:sff_vector_valued_components}, we write
\begin{equation}\label{eq:sff_vector_valued_first_derivative}
    h_{ijk}\doteq(\nabla^\perp_kB)(e_i,e_j)=\sum_\alpha h^\alpha_{ijk}\nu_\alpha,
    \qquad
    h_{ijkl}\doteq\bigl((\nabla^\perp)^2_{l,k}B\bigr)(e_i,e_j)=\sum_\alpha h^\alpha_{ijkl}\nu_\alpha .
\end{equation}
The components \eqref{eq:first_covariant_derivative_component_sff} obey the commutation rule
\begin{equation}\label{eq:sff_commutation_rule}
    h^\alpha_{ijkl}=h^\alpha_{ijlk}+\sum_m h^\alpha_{mj}R_{mikl}+\sum_m h^\alpha_{im}R_{mjkl}+ \sum_\beta h^\beta_{ij}R^\perp_{kl\beta\alpha}.
\end{equation}

\subsection{Fundamental equations}

The intrinsic and extrinsic geometry of $M$ are related by the following equations (see e.g.\ \cite{DoCarmo_RiemannianGeometry}, Chapter $6$).

\begin{itemize}
    \item \emph{Gauss' equation}: for $X,Y,Z,W\in\Gamma(TM)$,
    \begin{equation}\label{eq:gauss_equation_general}
        R(X,Y,Z,W)=\langle B(X,Z),B(Y,W)\rangle-\langle B(X,W),B(Y,Z)\rangle,
    \end{equation}
    that is, in components,
    \begin{equation}\label{eq:gauss_equation}
        R_{ijkl}=\sum_\beta h^\beta_{ik}h^\beta_{jl}-\sum_\beta h^\beta_{il}h^\beta_{jk}=\langle h_{ik},h_{jl}\rangle-\langle h_{il},h_{jk}\rangle.
    \end{equation}
    Tracing \eqref{eq:gauss_equation_general} according to \eqref{eq:ricci_trace_convention},
    \begin{equation}\label{eq:ricci_tensor_from_gauss_equation}
        \operatorname{Ric}(X,Z)=\langle\mathbf H,B(X,Z)\rangle-\sum_i\langle B(X,e_i),B(e_i,Z)\rangle.
    \end{equation}

    \item \emph{Weingarten's equation}: for $\xi\in\Gamma(NM)$ let $S_\xi:\Gamma(TM)\to\Gamma(TM)$ be the
    \emph{shape operator}, $\langle S_\xi(X),Y\rangle\doteq\langle B(X,Y),\xi\rangle$. Differentiating
    $\langle\xi,Y\rangle=0$ along $X\in\Gamma(TM)$ gives
    $\langle\overline\nabla_X\xi,Y\rangle=-\langle\xi,\overline\nabla_XY\rangle=-\langle S_\xi(X),Y\rangle$,
    so that, by \eqref{eq:normal_connection_definition},
    \begin{equation}\label{eq:weingarten_equation}
        \overline\nabla_X\xi=-S_\xi(X)+\nabla^\perp_X\xi .
    \end{equation}

    \item \emph{Codazzi's equation}: for $X,Y,Z\in\Gamma(TM)$,
    \begin{equation}\label{eq:codazzi_equation_general}
        (\nabla^\perp_XB)(Y,Z)=(\nabla^\perp_YB)(X,Z).
    \end{equation}
    In components, $h^\alpha_{ijk}=h^\alpha_{kji}$, which
    makes $h^\alpha_{ijk}$ totally symmetric in $i,j,k$.

    \item \emph{Ricci's equation}: for $X,Y\in\Gamma(TM)$ and $\eta,\xi\in\Gamma(NM)$,
    \begin{equation*}
        R^\perp(X,Y,\eta,\xi)=\langle[S_\xi,S_\eta]X,Y\rangle,
    \end{equation*}
    that is, in components,
    \begin{equation}\label{eq:ricci_equation}
        R^\perp_{ij\alpha\beta}=\sum_k\left(h^\alpha_{ik}h^\beta_{kj}-h^\alpha_{jk}h^\beta_{ki}\right)=[h^\alpha,h^\beta]_{ij}.
    \end{equation}
\end{itemize}

Codazzi's equation, together with the commutation rule \eqref{eq:sff_commutation_rule}, gives the rough Laplacian of the second fundamental form,
\begin{equation}\label{eq:sff_laplacian}
    \Delta^\perp h^\alpha_{ij}=\bigl\langle(\nabla^\perp)^2_{j,i}\mathbf H,\nu_\alpha\bigr\rangle+\sum_{k,m}h_{mk}^\alpha R_{mijk}+\sum_{k,m}h^\alpha_{im}R_{mkjk}+\sum_{k,\beta} h^\beta_{ik}R^\perp_{jk\beta\alpha},
\end{equation}
where $\Delta^\perp h^\alpha_{ij}$ stands for $\bigl\langle (\Delta^\perp B)(e_i,e_j),\nu_\alpha\bigr\rangle$.

\section{Translators as \texorpdfstring{$\rho$}{rho}-minimal immersions}\label{sec:translators}

Given its invariance under tangential reparametrizations, equation \eqref{eq:translator} is, as a system for the immersion $x$, degenerate, and it becomes a quasilinear elliptic system only after fixing a parametrization. As for self-shrinkers (see, e.g., \cite{IRR_sshc_classification}), the weighted formulation $\mathbf H_\rho\equiv 0$ of Subsection \ref{subsec:weighted_manifolds} does not remove the degeneracy but it makes clear that the curvature quantities of interest, namely $|\mathbf H|^2$, $|B|^2$, $|B_\nu|^2$, $|B_I|^2$ and the pinching function $f$ of Section \ref{sec:pinched_submanifolds}, satisfy scalar elliptic equations driven by the drift Laplacian $\Delta_\rho$.

This point of view has important consequences, because, for translators, $-\Delta_\rho$ is the exact elliptic counterpart of the heat operator $\partial_t-\Delta$ along the flow, so that the parabolic evolution equations of \cite{Naff_Planarity} and \cite{LeeNaffZhu_Planarity} can be easily transferred to the elliptic setting.

\subsection{Weighted manifolds, drift Laplacian and weighted-minimality}\label{subsec:weighted_manifolds}

A weighted manifold $N^m_\psi$ is a Riemannian manifold $(N^m,\langle\cdot,\cdot\rangle)$ equipped with the measure $d\operatorname{vol}_\psi\doteq e^{-\psi}d\operatorname{vol}_N$, where $\psi\in C^\infty(N)$.
The associated drift Laplacian (or weighted Laplacian or $\psi$-Laplacian) is
\begin{equation*}
    \Delta^N_\psi\varphi\doteq \operatorname{div}^N_\psi(\nabla^N \varphi)=\Delta^N \varphi-\langle\nabla^N \psi,\nabla^N \varphi\rangle,\qquad \operatorname{div}^N_\psi X\doteq e^\psi\operatorname{div}^N(e^{-\psi}X),
\end{equation*}
for $\varphi\in C^2(N)$. The operator $\Delta^N_\psi$ is linear and elliptic.

For an isometric immersion $x:M^n\to N^m_\psi$ the weighted mean curvature vector is $\mathbf H_\psi\doteq\mathbf H+(\overline\nabla \psi)^\perp$, and the immersion is $\psi$-minimal if $\mathbf H_\psi\equiv0$. Such an immersion induces the weighted structure $M^n_\psi\doteq(M,g,e^{-\psi\circ x}d\operatorname{vol}_M)$.

\subsection{Translators as \texorpdfstring{$\rho$}{rho}-minimal immersions and first-order identities}

Let $x:M^n\to\mathbb R^{n+p}$ be a translator satisfying $\mathbf H=\vec v^\perp$, with $|\vec v|=1$.
Setting $\rho(\cdot)\doteq-\langle \vec v,\cdot\rangle$ we have $\overline\nabla\rho=-\vec v$, hence $\nabla\rho=-\vec v^\top$ and $(\overline\nabla\rho)^\perp=-\vec v^\perp$. Thus $\mathbf H_\rho=\mathbf H-\vec v^\perp$ and translators are exactly the $\rho$-minimal immersions into $\mathbb R^{n+p}_\rho$. The drift Laplacian on $M^n_\rho$ reads as
\begin{equation}\label{eq:drift_laplacian_translators}
    \Delta_\rho \varphi=\Delta \varphi+\langle \vec v^\top,\nabla \varphi\rangle,\qquad \varphi\in C^2(M).
\end{equation}
We write, accordingly, $\Delta^\perp_\rho\doteq\Delta^\perp+\nabla^\perp_{\vec v^\top}$
for its counterpart on $(T^*M)^{\otimes k}\otimes NM$.

\begin{prop}
    Let $x:M^n\to\mathbb R^{n+p}$ be a translator. Then, for every $X\in\Gamma(TM)$,
    \begin{equation}\label{eq:V_top_derivative}
        \nabla_X\vec v^\top=S_{\mathbf H}(X),\qquad
        \nabla^\perp_X\mathbf H=-B(X,\vec v^\top).
    \end{equation}
    As a consequence, if $|\mathbf H|>0$, then $|\mathbf H|\langle\nabla^\perp_X\nu,\nu_\beta\rangle=-h^\beta(X,\vec v^\top)$
    for $\beta\geq n+2$. Moreover,
    \begin{equation}\label{eq:gradient_H_estimate}
        |\mathbf H| \bigl|\nabla^\perp\nu\bigr|\leq|B_I| |\vec v^\top|.
    \end{equation}
\end{prop}

\begin{proof}
    Since $\vec v$ is parallel in $\mathbb R^{n+p}$ we have $0=\overline\nabla_X(\vec v^\top)+\overline\nabla_X(\vec v^\perp)$.
    The first summand splits as
    $\nabla_X\vec v^\top+B(X,\vec v^\top)$, while by Weingarten's equation \eqref{eq:weingarten_equation}
    and $\mathbf H=\vec v^\perp$ the second one splits as $-S_{\mathbf H}(X)+\nabla^\perp_X\mathbf H$.
    Equating tangential and normal parts to zero gives \eqref{eq:V_top_derivative}.

    Projecting the second identity in \eqref{eq:V_top_derivative} onto $\mathcal I$ and comparing with \eqref{eq:nabla_perp_H_split} gives $|\mathbf H|\sigma^\beta_k=-h^\beta(e_k,\vec v^\top)$ for $\beta\geq n+2$. The estimate follows by Cauchy--Schwarz.
\end{proof}

\subsection{The Omori--Yau maximum principle for translators}

The proof of Theorem \ref{thm:codim_reduction_n_geq_5} uses a bounded function $\omega\geq0$, built in Section \ref{sec:codim_reduction_n_geq_5} out of $|B_I|^2$ and of the pinching function $f$ defined in \eqref{eq:pinching_function}, which vanishes identically if and only if $B_I\equiv0$. With $u\doteq|B_I|^2/f$ as in Section \ref{sec:pinched_submanifolds}, for every $\underline u>0$ it satisfies
\begin{equation*}
    \Delta_\rho\log\omega\ \geq\ \kappa_1|B_I|^2-\kappa_2\bigl|\nabla\log\omega\bigr|^2,
\end{equation*}
on $\{\omega>0\}\cap\{u\geq \underline u\}$, where $\kappa_1=42$ and $\kappa_2=\kappa_2(n,\varepsilon_0,\underline u)>0$. To prove that $\omega\equiv 0$ we reason by contradiction, assuming that $\omega^*\doteq \sup_M\omega>0$ and evaluating the inequality along a sequence of almost maximal points for $\omega$, so that $u\geq \underline u$ for some $\underline u=\underline u(n,\omega^*)>0$, and $\nabla\log\omega$ is almost zero. On a complete manifold with Ricci curvature bounded from below, the required sequence is provided by the Omori--Yau maximum principle, and Proposition \ref{prop:omori_yau_translators} shows that it holds for $\Delta_\rho$ in the present setting.

\begin{defn}\label{def:omori_yau}
    The weighted manifold $M^n_\rho$ satisfies the \emph{Omori--Yau maximum principle} for $\Delta_\rho$ if for every $\varphi\in C^2(M)$ with $\varphi^*\doteq\sup_M\varphi<+\infty$ there exists a sequence $\{q_k\}\subset M$ such that
    \begin{equation*}
        \varphi(q_k)>\varphi^*-\frac1k,\qquad |\nabla \varphi|(q_k)<\frac1k,\qquad \Delta_\rho\varphi(q_k)<\frac1k .
    \end{equation*}
\end{defn}

\begin{prop}\label{prop:omori_yau_translators}
    Let $x:M^n\to\mathbb R^{n+p}$ be a complete, $c$-quadratically pinched translator for some positive constant $c$. Then $\operatorname{Ric}\geq-(\sqrt{c}+c)g$,
    and $M^n_\rho$ satisfies the Omori--Yau maximum principle for $\Delta_\rho$.
\end{prop}
\begin{proof}
    By the pinching assumption and the bound $|\mathbf H|^2=|\vec v^\perp|^2\leq|\vec v|^2=1$, one has that $|B|^2\leq c|\mathbf H|^2\leq c$. For a unit vector $X$, \eqref{eq:ricci_tensor_from_gauss_equation} and Cauchy--Schwarz give $\operatorname{Ric}(X,X)\geq-|\mathbf H||B|-|B|^2\geq-(\sqrt{c}+c)$.

    Since $M$ is complete with Ricci curvature bounded from below, the classical Omori--Yau maximum principle holds for $\Delta$ (see, e.g., \cite[Theorem 2.3]{AliasMastroliaRigoli_MaximumPrinciples}). Given $\varphi$ bounded above there is $\{q_k\}\subset M$ with $\varphi(q_k)>\varphi^*-\frac{1}{2k}>\varphi^*-\frac 1k$, $|\nabla \varphi|(q_k)<\frac1{2k}$ and $\Delta \varphi(q_k)<\frac1{2k}$. By \eqref{eq:drift_laplacian_translators} and $|\vec v^\top|\leq1$,
    \begin{equation*}
        \Delta_\rho \varphi(q_k)=\Delta \varphi(q_k)+\langle \vec v^\top,\nabla \varphi\rangle(q_k)<\frac{1}{2k}+\frac{1}{2k}=\frac1k .
    \end{equation*}
\end{proof}

\section{Quadratically pinched submanifolds}\label{sec:pinched_submanifolds}

Throughout Sections \ref{sec:pinched_submanifolds}--\ref{sec:codim_reduction_n_geq_5}, we assume that
the immersion $x:M^n\to \mathbb R^{n+p}$ satisfies \eqref{eq:main_pinching}. It is not restrictive to assume that $0<\varepsilon_0<\frac{1}{3n}$. Indeed, if $\varepsilon_0=1/(3n)$, then the pinching condition and the identity
\[
|B|^2=|\mathring{B}|^2+\frac{1}{n}|H|^2
\]
imply that $\mathring{B}\equiv 0$. Thus $x$ is totally umbilical. By the classification of complete totally umbilical submanifolds of Euclidean space, $x(M)$ is either an affine $n$-plane or a round $n$-sphere contained in an $(n+1)$-dimensional affine subspace. The latter alternative is impossible, since closed translators do not exist. Hence $B\equiv 0$ and the second conclusion of Theorem \ref{thm:main_codim_reduction} is a straightforward consequence. On the other hand, if $\varepsilon_0>\frac{1}{3n}$ the pinching condition forces $B\equiv 0$, automatically giving the first conclusion of Theorem \ref{thm:main_codim_reduction}.

We define the pinching function as
\begin{equation}\label{eq:pinching_function}
    f\doteq \frac{4}{3n}|\mathbf H|^2-|B|^2,
\end{equation}
which trivially satisfies, on all $M$, the following chain of inequalities,
\begin{equation}\label{eq:f_upper_bound}
    0\leq \varepsilon_0|\mathbf H|^2\leq f\leq\frac{1}{3n}|\mathbf H|^2.
\end{equation}
On the open set $\{|\mathbf H|>0\}$, substituting \eqref{eq:norm_decompositions} into
\eqref{eq:pinching_function} yields the algebraic identity
\begin{equation}\label{eq:algebraic_identity}
    f+|B_I|^2+|\mathring B_\nu|^2=\frac{1}{3n}|\mathbf H|^2 .
\end{equation}

If $|\mathbf H|>0$, the strict positivity of $f$ ensures that the function
\begin{equation*}
    u \doteq \frac{|B_I|^2}{f}
\end{equation*}
is well defined. Dividing \eqref{eq:algebraic_identity} by $f$ and applying the lower bound $f\geq\varepsilon_0|\mathbf H|^2$, we obtain
\begin{equation}\label{eq:uniform_pinching_bounds_u}
    u\leq\frac{1}{3n}\frac{|\mathbf H|^2}{f}-1\leq\frac{1-3n\varepsilon_0}{3n\varepsilon_0}\doteq \overline u .
\end{equation}

\subsection{Derivatives in the adapted frame}

Assume that the mean curvature vector is nowhere vanishing. Since $\langle\nabla^\perp_k\nu,\nu\rangle=0$ and $\langle\nabla^\perp_k\nu,\nu_\beta\rangle=\sigma^\beta_k$ for $\beta\geq n+2$, the Leibniz rule for the connection of Subsection \ref{subsec:connections} applied to $B=\sum_\alpha h^\alpha\otimes\nu_\alpha$ gives
\begin{equation}\label{eq:levi_civita_vs_normal}
    h^{n+1}_{ijk}=\nabla_kh^{n+1}_{ij}-\sum_{\beta\geq n+2}\sigma^\beta_kh^\beta_{ij},
\end{equation}
where $\nabla_kh^{n+1}_{ij}$ is the Levi-Civita derivative of the symmetric $2$-tensor $h^{n+1}$. Moreover, $\sum_kh^\alpha_{ikk}=\sum_kh^\alpha_{kki}=\langle\nabla^\perp_i\mathbf H,\nu_\alpha\rangle$ by Codazzi's equation \eqref{eq:codazzi_equation_general}, so that, by \eqref{eq:nabla_perp_H_split}, for all $i$ and all $\beta\geq n+2$,
\begin{equation}\label{eq:codazzi_traces}
    \sum_kh^{n+1}_{ikk}=\nabla_i|\mathbf H|,\qquad \sum_kh^{\beta}_{ikk}=|\mathbf H| \sigma^\beta_i .
\end{equation}

\subsection{Gradient and quartic quantities}

Keep assuming $|\mathbf H|>0$ everywhere. Besides the quantities already introduced, we will need the following ones. First,
\begin{equation}\label{eq:C_E_definition}
    E_{ijk}\doteq h^{n+1}_{ijk}-\frac1n\delta_{ij}\nabla_k|\mathbf H| ,\qquad C_{ijk}\doteq\bigl\langle B_I(e_i,e_j),\nabla^\perp_k\nu\bigr\rangle=\sum_{\beta\geq n+2}\sigma^\beta_kh^\beta_{ij},
\end{equation}
with $|C|^2\doteq\sum_{i,j,k}C_{ijk}^2$ and $\langle E,C\rangle\doteq\sum_{i,j,k}E_{ijk}C_{ijk}$, and their counterparts in the directions of $\mathcal I$,
\begin{equation}\label{eq:E_I_D_definition}
    E^\gamma_{ijk}\doteq h^\gamma_{ijk}-\frac1n\delta_{ij}|\mathbf H| \sigma^\gamma_k,
    \qquad
    D^\gamma_{ijk}\doteq h^\gamma_{ijk}-h^{n+1}_{ij}\sigma^\gamma_k,
    \qquad \gamma\geq n+2,
\end{equation}
with $|E_I|^2\doteq\sum_{i,j,k}\sum_{\gamma\geq n+2}(E^\gamma_{ijk})^2$ and $|D|^2$ defined analogously. By \eqref{eq:tracefree_sff_components} and \eqref{eq:nabla_perp_H_split}, $E_{ijk}$ and $E^\gamma_{ijk}$ are, respectively, the $\nu$- and the $\mathcal I$-components of $\nabla^\perp\mathring B$, while by \eqref{eq:B_I_id_2} below the pair $(-C_{ijk},D^\gamma_{ijk})$ consists of the $\nu$- and the $\mathcal I$-components of $\nabla^\perp B_I$.
Second, we split the gradient and the quartic contractions of $B$ as
\begin{equation}\label{eq:gradient_pieces_notation}
    \mathcal N_\nu\doteq\sum_{i,j,k}(h^{n+1}_{ijk})^2,\qquad
    \mathcal N_I\doteq\sum_{i,j,k}\sum_{\gamma\geq n+2}(h^{\gamma}_{ijk})^2,\qquad
    \bigl|\nabla^\perp B\bigr|^2=\mathcal N_\nu+\mathcal N_I ,
\end{equation}
and
\begin{equation*}
    \begin{aligned}
        \mathcal S_I&\doteq\!\!\sum_{\alpha,\beta\geq n+2}\!\![\operatorname{tr}(h^\alpha h^\beta)]^2,\quad &
        \mathcal S_\nu&\doteq\!\!\sum_{\beta\geq n+2}\!\![\operatorname{tr}(h^{n+1} h^\beta)]^2, \\
        \mathcal K_I&\doteq\!\!\sum_{\alpha,\beta\geq n+2}\!\!|[h^\alpha,h^\beta]|^2, &
        \mathcal K_\nu&\doteq\!\!\sum_{\beta\geq n+2}\!\!|[h^{n+1},h^\beta]|^2,
    \end{aligned}
\end{equation*}
and we write $\mathcal R_I\doteq\mathcal S_I+\mathcal K_I+\mathcal K_\nu$. Setting
\begin{equation}\label{eq:S_K_definition}
    \mathcal S\doteq\sum_{\alpha,\beta}[\operatorname{tr}(h^\alpha h^\beta)]^2,\qquad
    \mathcal K\doteq\sum_{\alpha,\beta}|[h^\alpha,h^\beta]|^2,
\end{equation}
and using $\operatorname{tr}((h^{n+1})^2)=|B_\nu|^2$ and $[h^{n+1},h^{n+1}]=0$, one obtains
\begin{equation}\label{eq:splitting_quartic_contractions}
    \mathcal S=\mathcal S_I+2\mathcal S_\nu+|B_\nu|^4,
    \qquad
    \mathcal K=\mathcal K_I+2\mathcal K_\nu .
\end{equation}
Finally, we set
\begin{equation}\label{eq:definition_factor_Q_naff}
    Q_{ijk}\doteq\nabla_kh^{n+1}_{ij}-h^{n+1}_{ij}\nabla_k\log|\mathbf H| .
\end{equation}

\begin{lem}\label{lem:B_I_elementary_identities}
    With the above notation the following hold.
    \begin{enumerate}
        \item[(i)] For all $i,j,k$ and all $\gamma\geq n+2$, defining $|\nabla h^{n+1}|^2\doteq \sum_{i,j,k}(\nabla_k h^{n+1}_{ij})^2$,
        \begin{align}
            &\nabla_kh^{n+1}_{ij}=h^{n+1}_{ijk}+C_{ijk},
            &&|\nabla h^{n+1}|^2=\mathcal N_\nu+2\langle E,C\rangle+|C|^2,\label{eq:B_I_id_1}\\
            &\langle(\nabla^\perp_kB_I)(e_i,e_j),\nu\rangle=-C_{ijk},
            &&\langle(\nabla^\perp_kB_I)(e_i,e_j),\nu_\gamma\rangle=D^\gamma_{ijk},\label{eq:B_I_id_2}
        \end{align}
        and consequently $\bigl|\nabla^\perp B_I\bigr|^2=|C|^2+|D|^2$, that is,
        \begin{equation}\label{eq:norm_nabla_B_I}
            \bigl|\nabla^\perp B_I\bigr|^2=|C|^2+\mathcal N_I-2\sum_{i,j,k}h^{n+1}_{ij}\langle h_{ijk},\nabla^\perp_k\nu\rangle+|B_\nu|^2\bigl|\nabla^\perp\nu\bigr|^2 .
        \end{equation}
        \item[(ii)] The tensors $E$ and $E_I$ are trace-free in their first two indices, and
        \begin{equation}\label{eq:E_E_I_norms}
            |E|^2=\mathcal N_\nu-\frac1n\bigl|\nabla|\mathbf H|\bigr|^2,
            \qquad
            |E_I|^2=\mathcal N_I-\frac1n|\mathbf H|^2\bigl|\nabla^\perp\nu\bigr|^2 ,
        \end{equation}
        while $E^\gamma_{ijk}-D^\gamma_{ijk}=\mathring h^{n+1}_{ij}\sigma^\gamma_k$ for every $\gamma\geq n+2$.
        \item[(iii)] Setting $b_k\doteq\sum_{i,j}\mathring h^{n+1}_{ij}C_{ijk}$, one has $b_k=\sum_{\beta\geq n+2}\operatorname{tr}(h^{n+1}h^\beta)\sigma^\beta_k$ and
        \begin{equation}\label{eq:Q_C_contraction}
            \sum_{i,j,k}Q_{ijk}C_{ijk}=|C|^2+\langle E,C\rangle-\sum_k b_k\nabla_k\log|\mathbf H| .
        \end{equation}
    \end{enumerate}
\end{lem}
\begin{proof}
    The first identity in \eqref{eq:B_I_id_1} is \eqref{eq:levi_civita_vs_normal} rewritten through \eqref{eq:C_E_definition}. Since $\sum_iC_{iik}=0$, which is a consequence of $\operatorname{tr}h^\beta=0$ for $\beta\geq n+2$, it follows from \eqref{eq:C_E_definition} that $\sum_{i,j,k}h^{n+1}_{ijk}C_{ijk}=\langle E,C\rangle$. Then, squaring and summing the first identity in \eqref{eq:B_I_id_1} yields the second.

    Identities \eqref{eq:B_I_id_2} follow by differentiating $B_I=B-h^{n+1}\otimes\nu$, taking, respectively, the inner product with $\nu$ and $\nu_\gamma$, and using the first identity in \eqref{eq:B_I_id_1}. Summing their squares gives $\bigl|\nabla^\perp B_I\bigr|^2=|C|^2+|D|^2$, and expanding $|D|^2$ through \eqref{eq:E_I_D_definition} gives \eqref{eq:norm_nabla_B_I}.

    For (ii), the traces of $E$ and $E_I$ in their first two indices vanish by \eqref{eq:codazzi_traces}. Temporarily setting $E^{n+1}_{ijk}\doteq E_{ijk}$, by \eqref{eq:C_E_definition}, \eqref{eq:E_I_D_definition} and \eqref{eq:nabla_perp_H_split}, for every $\alpha$ one has $h^\alpha_{ijk}=E^\alpha_{ijk}+\frac1n\delta_{ij}\langle\nabla^\perp_k\mathbf H,\nu_\alpha\rangle$. Trace-freeness makes the two summands orthogonal for every fixed $\alpha$ and $k$, so
    \begin{equation*}
        \sum_{i,j}(h^\alpha_{ijk})^2=\sum_{i,j}(E^\alpha_{ijk})^2+\frac1n\langle\nabla^\perp_k\mathbf H,\nu_\alpha\rangle^2 .
    \end{equation*}
    Summing over $k$ for $\alpha=n+1$, and over $k$ and $\alpha\geq n+2$, and using \eqref{eq:nabla_perp_H_split}, gives \eqref{eq:E_E_I_norms}. The last statement of (ii) is a simple consequence of \eqref{eq:tracefree_sff_components}.    

    Finally, again because $\sum_iC_{iik}=0$,
    \begin{equation*}
        b_k=\sum_{i,j}\mathring h^{n+1}_{ij}C_{ijk}
        =\sum_{i,j}h^{n+1}_{ij}C_{ijk}
        =\sum_{\beta\geq n+2}\operatorname{tr}(h^{n+1}h^\beta)\sigma^\beta_k,
    \end{equation*}
    and \eqref{eq:Q_C_contraction} follows by contracting \eqref{eq:definition_factor_Q_naff} with $C_{ijk}$.
\end{proof}    

\begin{rem}\label{rem:Q_term_alternative_expression}
    Thanks to \eqref{eq:tracefree_sff_components}, the compatibility of $\nabla$ with the metric $g$ and \eqref{eq:B_I_id_1},
    \begin{equation*}
        \nabla_k\mathring h^{n+1}_{ij}=\nabla_kh^{n+1}_{ij}-\frac{1}{n}\delta_{ij}\nabla_k|\mathbf H|=h^{n+1}_{ijk}-\frac{1}{n}\delta_{ij}\nabla_k|\mathbf H|+C_{ijk}=E_{ijk}+C_{ijk}.
    \end{equation*}
    Therefore we recover the definition of $Q_{ijk}$ in \cite[Equation (3.19)]{Naff_Planarity},
    \begin{align*}
        Q_{ijk}&=\nabla_kh^{n+1}_{ij}-h^{n+1}_{ij}\nabla_k\log|\mathbf H|
        \\&=\nabla_k\mathring h^{n+1}_{ij}+\frac{1}{n}\delta_{ij}\nabla_k|\mathbf H|-\left(\mathring h^{n+1}_{ij}+\frac{1}{n}\delta_{ij}|\mathbf H|\right)\nabla_k\log|\mathbf H|
        \\&=\nabla_k\mathring h^{n+1}_{ij}-\mathring h^{n+1}_{ij}\nabla_k\log|\mathbf H|
        \\&=E_{ijk}+C_{ijk}-\mathring h^{n+1}_{ij}\nabla_k\log|\mathbf H|.
    \end{align*}
\end{rem}

\section{Simons' equations for translators}\label{section_simons_equations}

\subsection{Second covariant derivatives of \texorpdfstring{$\mathbf H$}{H} and Simons' equation for \texorpdfstring{$B$}{B}}

\begin{prop}
    Let $x:M^n\to\mathbb R^{n+p}$ be a translator. Then, for all $i,j$,
    \begin{align}
        (\nabla^\perp)^2_{j,i}\mathbf H&=-\sum_k\langle \mathbf H,h_{jk}\rangle h_{ik}-\sum_k\langle \vec v,e_k\rangle h_{ijk},\label{eq:normal_second_covariant_derivatives_H}
        \\ (\Delta^\perp_\rho B)(e_i,e_j)&=2\sum_{k,l}\langle h_{ik},h_{jl}\rangle h_{kl}-\sum_{k,l}\langle h_{ij},h_{kl}\rangle h_{kl}\label{eq:simons_equation_translators}
        \\&\qquad -\sum_{k,l}\langle h_{kl},h_{jl}\rangle h_{ik}-\sum_{k,l}\langle h_{ik},h_{kl}\rangle h_{jl} .\notag
    \end{align}
\end{prop}

\begin{proof}
    Fix a point $q\in M$ and choose a geodesic frame, i.e. $(\nabla_ie_j)(q)=0$. Differentiating the second identity in \eqref{eq:V_top_derivative}, at $q$
    \begin{align*}
        (\nabla^\perp)^2_{j,i}\mathbf H&=-\sum_k\nabla^\perp_j\left(\langle \vec v,e_k\rangle h_{ik}\right)
        \\&=-\sum_k e_j\left(\langle \vec v,e_k\rangle\right)h_{ik}-\sum_k \langle \vec v,e_k\rangle h_{ikj}
        \\&=-\sum_k\langle S_{\mathbf H}(e_j),e_k\rangle h_{ik}-\sum_k\langle \vec v,e_k\rangle h_{ijk}
        \\&=-\sum_k\langle \mathbf H, h_{jk}\rangle h_{ik}-\sum_k\langle \vec v,e_k\rangle h_{ijk}.
    \end{align*}
    We have used, in the second last equality, the first identity in \eqref{eq:V_top_derivative} together with Codazzi's equation. This concludes the proof of \eqref{eq:normal_second_covariant_derivatives_H}.

    To obtain \eqref{eq:simons_equation_translators}, we compute the three curvature terms of \eqref{eq:sff_laplacian} separately.
    By Gauss' equation \eqref{eq:gauss_equation},
    \begin{equation}\label{eq:simons_curvature_term_1}
        \sum_{k,m}h_{mk}R_{mijk}=\sum_{k,m}\langle h_{ik},h_{mj}\rangle h_{mk}-\sum_{k,m}\langle h_{ij},h_{mk}\rangle h_{mk}.
    \end{equation}
    By the trace convention \eqref{eq:ricci_trace_convention} one has $\sum_kR_{mkjk}=\operatorname{Ric}_{mj}$, and \eqref{eq:ricci_tensor_from_gauss_equation} gives $\operatorname{Ric}_{mj}=\langle\mathbf H,h_{mj}\rangle-\sum_k\langle h_{mk},h_{kj}\rangle$. Hence
    \begin{equation}\label{eq:simons_curvature_term_2}
        \sum_{k,m}h_{im}R_{mkjk}=\sum_m\langle\mathbf H,h_{mj}\rangle h_{im}-\sum_{k,m}\langle h_{mk},h_{kj}\rangle h_{im}.
    \end{equation}
    By Ricci's equation \eqref{eq:ricci_equation}, $R^\perp_{jk\beta\alpha}=\sum_m\left(h^\beta_{jm}h^\alpha_{mk}-h^\beta_{km}h^\alpha_{mj}\right)$, so that
    \begin{equation}\label{eq:simons_curvature_term_3}
        \sum_\alpha\left(\sum_{k,\beta}h^\beta_{ik}R^\perp_{jk\beta\alpha}\right)\nu_\alpha=\sum_{k,m}\langle h_{ik},h_{jm}\rangle h_{km}-\sum_{k,m}\langle h_{ik},h_{km}\rangle h_{jm}.
    \end{equation}

    Substituting \eqref{eq:normal_second_covariant_derivatives_H} and \eqref{eq:simons_curvature_term_1}--\eqref{eq:simons_curvature_term_3} into \eqref{eq:sff_laplacian}, $\sum_m\langle\mathbf H,h_{mj}\rangle h_{im}$ and $-\sum_k\langle\mathbf H,h_{jk}\rangle h_{ik}$ cancel, and moving $\sum_k\langle \vec v,e_k\rangle h_{ijk}=(\nabla^\perp_{\vec v^\top}B)(e_i,e_j)$ to the left-hand side yields \eqref{eq:simons_equation_translators}.
\end{proof}

\subsection{The drift Laplacians of \texorpdfstring{$|\mathbf H|^2$, $|B|^2$ and $f$}{|H|2, |B|2 and f}}

\begin{cor}\label{cor:simons_squared_norms_and_f}
    Let $x:M^n\to\mathbb R^{n+p}$ be a translator, and let $f$ be as in \eqref{eq:pinching_function}. Then the following hold.
    \begin{enumerate}
        \item[(i)] The mean curvature vector satisfies
        \begin{equation}\label{eq:simons_H}
            \Delta^\perp_\rho\mathbf H=-\sum_{k,l}\langle\mathbf H,h_{kl}\rangle h_{kl},
        \end{equation}
        and consequently
        \begin{equation}\label{eq:rho_laplacian_squared_norm_H_translators}
            \Delta_\rho|\mathbf H|^2=2|\nabla^\perp\mathbf H|^2-2|S_{\mathbf H}|^2 .
        \end{equation}
        \item[(ii)] The squared norm of the second fundamental form satisfies
        \begin{equation}\label{eq:rho_laplacian_squared_norm_B_translators}
            \Delta_\rho|B|^2=2\bigl|\nabla^\perp B\bigr|^2-2\mathcal S-2\mathcal K .
        \end{equation}
        \item[(iii)] The pinching function satisfies
        \begin{equation}\label{eq:rho_laplacian_f_identity}
            \Delta_\rho f=-2(\mathcal R_f+\mathcal G_f),
        \end{equation}
        where
        \begin{equation}\label{eq:reaction_gradient_f_definition}
            \mathcal R_f\doteq\frac{4}{3n}|S_{\mathbf H}|^2-\mathcal S-\mathcal K,
            \qquad
            \mathcal G_f\doteq\bigl|\nabla^\perp B\bigr|^2-\frac{4}{3n}|\nabla^\perp\mathbf H|^2 .
        \end{equation}
    \end{enumerate}
\end{cor}
\begin{proof}
    (i) We sum over $i=j$ in \eqref{eq:simons_equation_translators}, using that $\operatorname{tr}_g$ commutes with $\Delta^\perp_\rho$. After relabelling dummy indices, the first, third and fourth terms cancel out. Hence only the second term survives, giving \eqref{eq:simons_H}. Then
    \begin{equation*}
        \Delta_\rho|\mathbf H|^2=2|\nabla^\perp\mathbf H|^2+2\langle\Delta^\perp_\rho\mathbf H,\mathbf H\rangle=2|\nabla^\perp\mathbf H|^2-2\sum_{k,l}\langle\mathbf H,h_{kl}\rangle^2,
    \end{equation*}
    and $|S_{\mathbf H}|^2=\sum_{k,l}\langle\mathbf H,h_{kl}\rangle^2$ by definition of $S_{\mathbf H}$.

    (ii) Since $\Delta_\rho|B|^2=2\bigl|\nabla^\perp B\bigr|^2+2\langle\Delta^\perp_\rho B,B\rangle$, we contract \eqref{eq:simons_equation_translators} with $h_{ij}$,
    \begin{align*}
        \Delta_\rho|B|^2&=2\bigl|\nabla^\perp B\bigr|^2+4\sum_{i,j,k,l}\langle h_{ik}, h_{jl}\rangle\langle h_{kl},h_{ij}\rangle-2\sum_{i,j,k,l}\langle h_{ij},h_{kl}\rangle^2
        \\&\qquad -2\sum_{i,j,k,l}\langle h_{kl},h_{jl}\rangle\langle h_{ik},h_{ij}\rangle-2\sum_{i,j,k,l}\langle h_{ik},h_{kl}\rangle\langle h_{jl},h_{ij}\rangle
        \\&=2\bigl|\nabla^\perp B\bigr|^2+4\sum_{i,j,k,l}\langle h_{ik}, h_{jl}\rangle\langle h_{kl},h_{ij}\rangle-2\sum_{i,j,k,l}\langle h_{ij},h_{kl}\rangle^2-4\sum_{i,j,k,l}\langle h_{kl},h_{jl}\rangle\langle h_{ik},h_{ij}\rangle
        \\&=2\bigl|\nabla^\perp B\bigr|^2+4\sum_{\alpha,\beta}\operatorname{tr}((h^\alpha h^\beta)^2)-2\sum_{\alpha,\beta}[\operatorname{tr}(h^\alpha h^\beta)]^2-4\sum_{\alpha,\beta}\operatorname{tr}[(h^\alpha)^2 (h^\beta)^2]
        \\&=2\bigl|\nabla^\perp B\bigr|^2-2\mathcal S-2\mathcal K.
    \end{align*}
    The second equality follows from relabelling the indices in the final term of the second line, while the third simply rewrites the previous step in matrix notation. The last line follows from the identity $|[h^\alpha, h^\beta]|^2 = 2\operatorname{tr}[(h^\alpha)^2 (h^\beta)^2] - 2\operatorname{tr}((h^\alpha h^\beta)^2)$ together with \eqref{eq:S_K_definition}.

    (iii) Subtract \eqref{eq:rho_laplacian_squared_norm_B_translators} from $\frac{4}{3n}$ times \eqref{eq:rho_laplacian_squared_norm_H_translators}.
\end{proof}

\subsection{The drift Laplacians of \texorpdfstring{$|B_\nu|^2$}{|Bnu|2} and \texorpdfstring{$|B_I|^2$}{|BI|2}}

Throughout this subsection we assume that $|\mathbf H|>0$, so that the splitting $NM=\mathbb R\nu\oplus\mathcal I$ of Subsection \ref{subsec:sff} is available.

\begin{lem}
    Let $x:M^n\to\mathbb R^{n+p}$ be a translator with $|\mathbf H|>0$. Then, for all $i,j$ and all $\beta\geq n+2$,
    \begin{align}
        (\Delta h^{n+1})_{ij}&=\Delta^\perp h^{n+1}_{ij}+2\sum_k\bigl\langle h_{ijk},\nabla^\perp_k\nu\bigr\rangle+\bigl\langle h_{ij},\Delta^\perp\nu\bigr\rangle,\label{eq:delta_h_vs_delta_perp_h}\\
        \bigl\langle\Delta^\perp_\rho\nu,\nu_\beta\bigr\rangle&=-\operatorname{tr}(h^{n+1}h^\beta)-2\sum_k\nabla_k\log|\mathbf H| \sigma^\beta_k,\label{eq:drift_laplacian_nu}
        \\ \langle\Delta^\perp_\rho\nu,\nu\rangle&=-\bigl|\nabla^\perp\nu\bigr|^2.\label{eq:drift_laplacian_nu_nu_component}
    \end{align}
\end{lem}
\begin{proof}
    Fix a point $q\in M$ and a geodesic tangent frame in $q$, i.e. satisfying $(\nabla_ie_j)(q)=0$.
    Differentiating $h^{n+1}=\langle B,\nu\rangle$ twice gives, at $q$,
    \begin{equation*}
        \nabla_l\nabla_kh^{n+1}_{ij}=\bigl\langle h_{ijkl},\nu\bigr\rangle+\bigl\langle h_{ijk},\nabla^\perp_l\nu\bigr\rangle+\bigl\langle h_{ijl},\nabla^\perp_k\nu\bigr\rangle+\bigl\langle h_{ij},(\nabla^\perp)^2_{l,k}\nu\bigr\rangle.
    \end{equation*}
    Tracing over $k=l$ yields \eqref{eq:delta_h_vs_delta_perp_h}.

    For \eqref{eq:drift_laplacian_nu}, expand $\mathbf H=|\mathbf H|\nu$ as
    \begin{equation*}
        \Delta^\perp_\rho\mathbf H=(\Delta_\rho|\mathbf H|) \nu+2\sum_k\nabla_k|\mathbf H| \nabla^\perp_k\nu+|\mathbf H| \Delta^\perp_\rho\nu.
    \end{equation*}
    Then, the inner product with $\nu_\beta$, for some $\beta\geq n+2$, is
    \begin{equation*}
        \langle \Delta^\perp_\rho\mathbf H,\nu_\beta\rangle=2\sum_k\nabla_k|\mathbf H| \sigma^\beta_k+|\mathbf H| \langle\Delta^\perp_\rho\nu,\nu_\beta\rangle.
    \end{equation*}
    Comparing the $\nu_\beta$-components with those of \eqref{eq:simons_H}, which are $-|\mathbf H|\operatorname{tr}(h^{n+1}h^\beta)$, yields \eqref{eq:drift_laplacian_nu}.

    The last statement follows from $\Delta|\nu|^2=0$ and $\langle\nabla^\perp_{\vec v^\top}\nu,\nu\rangle=0$.
\end{proof}

\begin{lem}
    Let $x:M^n\to\mathbb R^{n+p}$ be a translator with $|\mathbf H|>0$. Then
    \begin{equation}\label{eq:rho_laplacian_squared_norm_B_nu}
        \Delta_\rho|B_\nu|^2=2|\nabla h^{n+1}|^2-2\mathcal K_\nu-4\mathcal S_\nu-2|B_\nu|^4+\mathcal D_\nu,
    \end{equation}
    where
    \begin{equation}\label{eq:defect_term_C}
        \mathcal D_\nu\doteq 4\sum_{i,j,k}h^{n+1}_{ij}\bigl\langle h_{ijk},\nabla^\perp_k\nu\bigr\rangle-2|B_\nu|^2\bigl|\nabla^\perp\nu\bigr|^2-4\sum_k b_k\nabla_k\log|\mathbf H|.
    \end{equation}
\end{lem}
\begin{proof}
    Since $|B_\nu|^2=|h^{n+1}|^2$, expanding the $\rho$-Laplacian of $|B_\nu|^2$ gives
    \begin{equation}\label{eq:bochner_B_nu}
        \Delta_\rho|B_\nu|^2=2|\nabla h^{n+1}|^2+2\bigl\langle h^{n+1},\Delta h^{n+1}\bigr\rangle+\bigl\langle \vec v^\top,\nabla|B_\nu|^2\bigr\rangle .
    \end{equation}
    Contracting \eqref{eq:delta_h_vs_delta_perp_h} with $h^{n+1}_{ij}$ yields
    \begin{equation}\label{eq:pairing_delta_h}
        \langle h^{n+1},\Delta h^{n+1}\rangle=\sum_{i,j}h^{n+1}_{ij}\Delta^\perp h^{n+1}_{ij}+2\sum_{i,j,k}h^{n+1}_{ij}\langle h_{ijk},\nabla^\perp_k\nu\rangle+\sum_\beta\operatorname{tr}(h^{n+1}h^\beta)\langle\Delta^\perp\nu,\nu_\beta\rangle,
    \end{equation}
    while contracting the first identity in \eqref{eq:B_I_id_1} with $2\langle \vec v,e_k\rangle h^{n+1}_{ij}$ and using, as established in Lemma \ref{lem:B_I_elementary_identities}(iii), that $b_k=\sum_{\beta\geq n+2}\operatorname{tr}(h^{n+1}h^\beta)\sigma^\beta_k$,
    \begin{equation*}
        \langle \vec v^\top,\nabla|B_\nu|^2\rangle=2\sum_{i,j,k}\langle \vec v,e_k\rangle h^{n+1}_{ij}h^{n+1}_{ijk}+2\sum_\beta\operatorname{tr}(h^{n+1}h^\beta)\langle\nabla^\perp_{\vec v^\top}\nu,\nu_\beta\rangle.
    \end{equation*}
    To conclude the proof, it remains to compute $\sum_{i,j}h^{n+1}_{ij}\Delta^\perp h^{n+1}_{ij}$. If we set $\alpha = n + 1$ in \eqref{eq:sff_laplacian}, contract the result with $h^{n+1}_{ij}$, and use the $\nu$-component of \eqref{eq:normal_second_covariant_derivatives_H} as well as equations \eqref{eq:gauss_equation}, \eqref{eq:ricci_tensor_from_gauss_equation}, and \eqref{eq:ricci_equation}, the four contributions evaluate to:
    \begin{align*}
        &-|\mathbf H|\operatorname{tr}((h^{n+1})^3)-\sum_{i,j,k}\langle \vec v,e_k\rangle h^{n+1}_{ij}h^{n+1}_{ijk},
        &&\sum_\beta\operatorname{tr}((h^{n+1}h^\beta)^2)-\sum_\beta[\operatorname{tr}(h^{n+1}h^\beta)]^2,\\
        &|\mathbf H|\operatorname{tr}((h^{n+1})^3)-\sum_\beta\operatorname{tr}[(h^{n+1})^2(h^\beta)^2],
        &&\sum_\beta\operatorname{tr}((h^{n+1}h^\beta)^2)-\sum_\beta\operatorname{tr}[(h^{n+1})^2(h^\beta)^2].
    \end{align*}
    The terms $\mp|\mathbf H|\operatorname{tr}((h^{n+1})^3)$ cancel and, since $|[h^{n+1},h^\beta]|^2=2\operatorname{tr}[(h^{n+1})^2(h^\beta)^2]-2\operatorname{tr}((h^{n+1}h^\beta)^2)$, splitting off the summand $\beta=n+1$ gives
    \begin{equation}\label{eq:sum_h_delta_perp_h}
        \sum_{i,j}h^{n+1}_{ij}\Delta^\perp h^{n+1}_{ij}=-\sum_{i,j,k}\langle \vec v,e_k\rangle h^{n+1}_{ij}h^{n+1}_{ijk}-\mathcal K_\nu-\mathcal S_\nu-|B_\nu|^4 .
    \end{equation}
    Substituting \eqref{eq:pairing_delta_h}--\eqref{eq:sum_h_delta_perp_h} into \eqref{eq:bochner_B_nu}, the terms containing $\langle \vec v,e_k\rangle$ cancel and the two $\nu$-derivative contributions combine into $2\sum_\beta\operatorname{tr}(h^{n+1}h^\beta)\langle\Delta^\perp_\rho\nu,\nu_\beta\rangle$. The final expression for $\mathcal D_\nu$ follows from \eqref{eq:drift_laplacian_nu} and \eqref{eq:drift_laplacian_nu_nu_component}.
\end{proof}

\begin{prop}
    Let $x:M^n\to\mathbb R^{n+p}$ be a translator with $|\mathbf H|>0$. Then
    \begin{equation}\label{eq:rho_laplacian_squared_norm_B_I}
        \Delta_\rho|B_I|^2=2\bigl|\nabla^\perp B_I\bigr|^2-2\mathcal R_I-4\sum_{i,j,k}Q_{ijk}C_{ijk}.
    \end{equation}
\end{prop}
\begin{proof}
    Subtracting \eqref{eq:rho_laplacian_squared_norm_B_nu} from \eqref{eq:rho_laplacian_squared_norm_B_translators} and using \eqref{eq:splitting_quartic_contractions}, the terms $\pm 4\mathcal S_\nu$ and $\pm2|B_\nu|^4$ cancel, and
    \begin{align*}
        \Delta_\rho|B_I|^2&=2\left(\bigl|\nabla^\perp B\bigr|^2-|\nabla h^{n+1}|^2\right)-2\mathcal R_I-\mathcal D_\nu
        \\&=2\mathcal N_I-4\langle E,C\rangle-2|C|^2-2\mathcal R_I-4\sum_{i,j,k}h^{n+1}_{ij}\langle h_{ijk},\nabla^\perp_k\nu\rangle
        \\&\qquad+2|B_\nu|^2\bigl|\nabla^\perp\nu\bigr|^2+4\sum_kb_k\nabla_k\log|\mathbf H| ,
    \end{align*}
    where the second equality is a consequence of \eqref{eq:gradient_pieces_notation}, \eqref{eq:B_I_id_1} and \eqref{eq:defect_term_C}. Using \eqref{eq:norm_nabla_B_I} and \eqref{eq:Q_C_contraction}, the expression for $\Delta_\rho|B_I|^2$ matches \eqref{eq:rho_laplacian_squared_norm_B_I}.
\end{proof}

\begin{rem}
    Under the notational correspondence established at the beginning of Section \ref{sec:estimates}, equation \eqref{eq:rho_laplacian_squared_norm_B_I} matches exactly \cite[Equation (3.18)]{Naff_Planarity} once the substitution $\partial_t-\Delta=-\Delta_\rho$ justified in Section \ref{sec:translators} is performed. The tensor $Q_{ijk}$ is, by Remark \ref{rem:Q_term_alternative_expression}, the one of \cite[Equation (3.19)]{Naff_Planarity}.
\end{rem}

\section{Algebraic and gradient estimates}\label{sec:estimates}

We collect, in the notation of Section \ref{sec:pinched_submanifolds}, the algebraic and gradient estimates of \cite{Naff_Planarity} specialized to the borderline case where the pinching constant is $\frac{4}{3n}$, which will be used in Section \ref{sec:codim_reduction_n_geq_5}. They are pointwise estimates that do not involve the translator equation. We state them without proof, referring the reader to \cite{Naff_Planarity}. The section closes with the differential inequality satisfied by the pinching function $f$.

The correspondence dictionary between Naff's notation and ours is: $|\langle\hat A_{ij},\hat A_{kl}\rangle|^2=\mathcal S_I$, $|\hat R^\perp|^2=\mathcal K_I$, $|R^\perp_{ij}(\nu_1)|^2=\mathcal K_\nu$, $\langle\hat A_{ij},\nabla^\perp_k\nu_1\rangle=C_{ijk}$ and, using that $\operatorname{tr}\:h^\beta=0$ for $\beta\geq n+2$, $\left|\sum_{i,j=1}^n\mathring h_{ij}\hat A_{ij} \right|^2=\mathcal S_\nu$.

\subsection{Algebraic estimates}

\begin{lem}[\cite{AndrewsBaker_MCFPinchedSubmanifolds}, {\cite[Theorem 1]{LiLi_RigidityMinimalSubmanifoldsSphere}}; {\cite[Lemma 4.1]{Naff_Planarity}}]
    At any point where $|\mathbf H|>0$,
    \begin{equation}\label{eq:naff_lemma_4_1}
        \mathcal S_\nu+\mathcal K_\nu\leq 2|\mathring B_\nu|^2|B_I|^2,
        \qquad
        \mathcal S_I+\mathcal K_I\leq \frac32|B_I|^4 .
    \end{equation}
    In particular, discarding $\mathcal S_\nu\geq 0$,
    \begin{equation}\label{eq:R_I_upper_bound}
        \mathcal R_I\leq\frac32|B_I|^4+2|\mathring B_\nu|^2|B_I|^2 .
    \end{equation}
\end{lem}

\begin{lem}[{\cite[(4.13)]{Naff_Planarity}}]\label{lem:non_negativity_Rf}
    Under the pinching assumption \eqref{eq:main_pinching}, at every point where $|\mathbf H|>0$ the reaction term $\mathcal R_f$ satisfies
    \begin{equation}\label{eq:reaction_lower_bound_f}
        \mathcal R_f\geq 6f|B_I|^2+4f|\mathring B_\nu|^2+3f^2\geq 0 .
    \end{equation}
    Moreover $\mathcal R_f=0$ at every point where $|\mathbf H|=0$.
\end{lem}

\subsection{Gradient estimates}

\begin{lem}[\cite{Huisken_FlowConvexSurfacesIntoSpheres}]
    For any immersed submanifold $M^n$ of $\mathbb R^{n+p}$, the gradient estimate
    \begin{equation}\label{eq:gradient_estimate_huisken}
        \bigl|\nabla^\perp B\bigr|^2\geq\frac{3}{n+2}\bigl|\nabla^\perp \mathbf H\bigr|^2
    \end{equation}
    holds pointwise.
\end{lem}

\begin{lem}[{\cite[(4.19)--(4.22)]{Naff_Planarity}}]
    Let $|\mathbf H|>0$. Then
    \begin{align}
        \frac{3}{n+2}\bigl|\nabla|\mathbf H|\bigr|^2&\leq \mathcal N_\nu, &
        \frac{3}{n+2}|\mathbf H|^2\bigl|\nabla^\perp\nu\bigr|^2&\leq \mathcal N_I, \label{eq:naff_4_19_4_20}\\
        \frac{2(n-1)}{n(n+2)}\bigl|\nabla|\mathbf H|\bigr|^2&\leq |E|^2, &
        \frac{2(n-1)}{n(n+2)}|\mathbf H|^2\bigl|\nabla^\perp\nu\bigr|^2&\leq |E_I|^2 .
        \label{eq:naff_4_21_4_22}
    \end{align}
\end{lem}

\begin{rem}
    By \eqref{eq:E_E_I_norms} the two inequalities in \eqref{eq:naff_4_21_4_22} are equivalent to those in \eqref{eq:naff_4_19_4_20}, since $\frac{2(n-1)}{n(n+2)}+\frac1n=\frac{3}{n+2}$. Both forms will be used in Section \ref{sec:codim_reduction_n_geq_5}.
\end{rem}

\subsection{The differential inequality for the pinching function}

\begin{prop}\label{prop:rho_laplacian_f}
    Let $x:M^n\to\mathbb R^{n+p}$ be a connected translator without boundary satisfying the assumption \eqref{eq:main_pinching}. Then the pinching function $f$ of \eqref{eq:pinching_function} satisfies
    \begin{equation}\label{eq:rho_laplacian_f}
        \Delta_\rho f\leq-\frac{2(5n-8)}{3n(n+2)}|\nabla^\perp\mathbf H|^2-2\mathcal R_f\leq0 .
    \end{equation}
    Consequently, either $f\equiv0$ or $f>0$ everywhere on $M^n$. In the former case, $|\mathbf H|\equiv 0$ and $B\equiv 0$. In the latter case, $|\mathbf H|>0$ everywhere.
\end{prop}
\begin{proof}
    By Corollary \ref{cor:simons_squared_norms_and_f}, $\Delta_\rho f=-2(\mathcal R_f+\mathcal G_f)$. The gradient estimate \eqref{eq:gradient_estimate_huisken} gives
    \begin{equation*}
        \mathcal G_f=\bigl|\nabla^\perp B\bigr|^2-\frac{4}{3n}|\nabla^\perp\mathbf H|^2\geq\left(\frac{3}{n+2}-\frac{4}{3n}\right)|\nabla^\perp\mathbf H|^2=\frac{5n-8}{3n(n+2)}|\nabla^\perp\mathbf H|^2\geq 0 ,
    \end{equation*}
    which proves the first inequality in \eqref{eq:rho_laplacian_f}. Moreover, $\mathcal R_f\geq 0$ by Lemma \ref{lem:non_negativity_Rf}, which proves the second inequality in \eqref{eq:rho_laplacian_f}.

    Since $f\geq0$ and $\Delta_\rho f\leq0$, as a consequence of the strong maximum principle for elliptic operators (see e.g.\ \cite[Theorem 3.5]{GilbardTrudinger_EllipticPDEs}), either $f\equiv 0$ or $f>0$ everywhere on $M$. In the first case, $0\equiv f\geq \varepsilon_0|\mathbf H|^2\geq 0$, hence $|\mathbf H|\equiv0$ and $B\equiv0$ by the pinching condition. If $f>0$ then $|\mathbf H|^2\geq\frac{3n}{4}f>0$.
\end{proof}

\section{Codimension reduction in dimensions \texorpdfstring{$n\geq 5$}{n>=5}}\label{sec:codim_reduction_n_geq_5}

Throughout this section $n\geq5$ and $x:M^n\to\mathbb R^{n+p}$ is a complete translator satisfying \eqref{eq:main_pinching}. 
We fix
\begin{equation*}
    \theta\doteq\frac56 .
\end{equation*}
For every $n\geq5$, this value satisfies the conditions \eqref{eq:condition_C1_new}--\eqref{eq:condition_C3_new} of Appendix \ref{sec:algebraic_lemma_5_6}, see Remark \ref{rem:conditions_C123_verified}. Where $|\mathbf H|>0$, the function $u=|B_I|^2/f$ is defined, and we set
\begin{equation*}
    \omega\doteq f u^{\frac{1}{1-\theta}}=\frac{|B_I|^{12}}{f^{5}} .
\end{equation*}
Since $|B_I|^2$ and $f>0$ are smooth on $\{|\mathbf H|>0\}$, so is $\omega$. At the points where, moreover, $|B_I|>0$, we set
\begin{equation*}
    \zeta\doteq(1-\theta)\nabla\log\omega=\nabla\log|B_I|^2-\theta\nabla\log f .
\end{equation*}
Finally, we regard the quantity $b_k=\sum_{i,j}\mathring h^{n+1}_{ij}C_{ijk}$ of Lemma \ref{lem:B_I_elementary_identities}(iii) as the vector field $b\doteq\sum_kb_ke_k$, so that $\langle\nabla\log|\mathbf H|,b\rangle=\sum_kb_k\nabla_k\log|\mathbf H|$.

\subsection{Lee-Naff-Zhu approach}

Substituting \eqref{eq:rho_laplacian_squared_norm_B_I} and \eqref{eq:rho_laplacian_f_identity}, at every point where $|\mathbf H|>0$,
\begin{equation}\label{eq:interpolated_expansion}
    \Delta_\rho|B_I|^2-\theta u \Delta_\rho f
    =2\left[\theta u \mathcal R_f-\mathcal R_I\right]
    +2\Bigl[\bigl|\nabla^\perp B_I\bigr|^2+\theta u \mathcal G_f-2\sum_{i,j,k}Q_{ijk}C_{ijk}\Bigr].
\end{equation}

In \cite[Appendix B.2]{LeeNaffZhu_Planarity} the two brackets of \eqref{eq:interpolated_expansion} are estimated separately and shown to be nonnegative pointwise. The term $2\langle\nabla\log|\mathbf H|,b\rangle$ contained in $-2\sum_{i,j,k}Q_{ijk}C_{ijk}$ is absorbed by Young's inequality, see \cite[(B.3)--(B.4)]{LeeNaffZhu_Planarity}. When $\theta<1$, the same argument closes if
\begin{equation*}
    r\doteq\frac{n+2}{3(n-1)}\leq\frac{\theta}{1+\theta}.
\end{equation*}
This needs $r<\frac12$, which holds if and only if $n\geq8$. If $\frac{4}{3n}$ is replaced by a pinching constant $c\in\left(\frac1n,\frac{4}{3n}\right]$, then $r=\frac{(n+2)(nc-1)}{n-1}$, and $r<\frac12$ becomes $c<\frac{3(n+1)}{2n(n+2)}$. This is the constant of \cite[Appendix B.2]{LeeNaffZhu_Planarity} entering $c_n^{\mathrm{Naff}}$. 

The argument below never requires the gradient bracket of \eqref{eq:interpolated_expansion} to be nonnegative.
We keep the gradient terms of \eqref{eq:interpolated_expansion} and those produced by the logarithmic derivative of $\omega$ together, in the quantity $\mathcal G_\zeta$ of \eqref{eq:G_zeta_definition}. We only prove the lower bound $\mathcal G_\zeta\geq-\widetilde C|B_I|^2|\zeta|^2$ at the points where $u\geq \underline u$ (Lemma \ref{lem:gradient_estimate_5_6}). This is enough, because $u\geq \underline u$ and $|\zeta|\to0$ along the Omori--Yau sequence. In the proof of this bound, the term $2\langle\nabla\log|\mathbf H|,b\rangle$ is rewritten by differentiating \eqref{eq:algebraic_identity}, and the term $\theta(1-\theta)|\nabla\log f|^2$ is kept.

\subsection{Reaction and gradient terms lower bounds}

\begin{lem}\label{lem:reaction_inequality_n_geq_7}
    At every point where $|\mathbf H|>0$,
    \begin{equation}\label{eq:reaction_inequality_n_geq_7}
        \theta u \mathcal R_f-\mathcal R_I
        \geq\frac72|B_I|^4+\frac43|\mathring B_\nu|^2|B_I|^2
        \geq\frac72|B_I|^4 .
    \end{equation}
\end{lem}
\begin{proof}
    Multiplying \eqref{eq:reaction_lower_bound_f} by $\theta u$, using $uf=|B_I|^2$ and discarding the term $\frac52f|B_I|^2\geq 0$,
    \begin{equation*}
        \theta u \mathcal R_f\geq5|B_I|^4+\frac{10}{3}|\mathring B_\nu|^2|B_I|^2 .
    \end{equation*}
    Subtracting \eqref{eq:R_I_upper_bound} gives \eqref{eq:reaction_inequality_n_geq_7}.
\end{proof}

\begin{lem}\label{lem:LNZ_B1_B2}
    At every point where $|\mathbf H|>0$,
    \begin{align}
        \bigl|\nabla^\perp B_I\bigr|^2&\geq|C|^2+\left[\frac{3(n-1)}{n+2}\left(f+|B_I|^2\right)+\frac{2n-5}{n+2}|\mathring B_\nu|^2\right]\bigl|\nabla^\perp\nu\bigr|^2,\label{eq:gradient_1_5_6}\\
        u \mathcal G_f&\geq\frac{5n-8}{n+2}\left[(1+u)|B_I|^2+u|\mathring B_\nu|^2\right]\bigl|\nabla^\perp\nu\bigr|^2+\frac{5n-8}{6(n-1)} u |E|^2 .\label{eq:gradient_2_5_6}
    \end{align}
\end{lem}
\begin{proof}
    Throughout the proof we use \eqref{eq:algebraic_identity} in the form $|\mathbf H|^2=3n(f+|B_I|^2+|\mathring B_\nu|^2)$.

    By Lemma \ref{lem:B_I_elementary_identities}, $\bigl|\nabla^\perp B_I\bigr|^2=|C|^2+|D|^2$ and $E^\gamma_{ijk}-D^\gamma_{ijk}=\mathring h^{n+1}_{ij}\sigma^\gamma_k$. The elementary inequality $|\Xi-\Psi|^2\geq\frac12|\Xi|^2-|\Psi|^2$ gives $|D|^2\geq\frac12|E_I|^2-|\mathring B_\nu|^2\bigl|\nabla^\perp\nu\bigr|^2$. Bounding $|E_I|^2$ from below by the second inequality in \eqref{eq:naff_4_21_4_22}, and using $\frac{n-1}{n(n+2)}|\mathbf H|^2=\frac{3(n-1)}{n+2}(f+|B_I|^2+|\mathring B_\nu|^2)$ and $\frac{3(n-1)}{n+2}-1=\frac{2n-5}{n+2}$, yields \eqref{eq:gradient_1_5_6}. This is $|C|^2$ plus one half of \cite[(4.23)]{Naff_Planarity}, and it is \cite[(B.1)]{LeeNaffZhu_Planarity} for the pinching constant $\frac{4}{3n}$.

    By \eqref{eq:gradient_pieces_notation}, \eqref{eq:nabla_perp_H_split} and the first identity in \eqref{eq:E_E_I_norms},
    \begin{equation*}
        \mathcal G_f=|E|^2-\frac{1}{3n}\bigl|\nabla|\mathbf H|\bigr|^2+\mathcal N_I-\frac{4}{3n}|\mathbf H|^2\bigl|\nabla^\perp\nu\bigr|^2 .
    \end{equation*}
    The first inequality in \eqref{eq:naff_4_21_4_22} gives $\frac{1}{3n}\bigl|\nabla|\mathbf H|\bigr|^2\leq\frac{n+2}{6(n-1)}|E|^2$, and $1-\frac{n+2}{6(n-1)}=\frac{5n-8}{6(n-1)}$. The second inequality in \eqref{eq:naff_4_19_4_20} gives
    \begin{equation*}
        \mathcal N_I-\frac{4}{3n}|\mathbf H|^2\bigl|\nabla^\perp\nu\bigr|^2\geq\frac{5n-8}{3n(n+2)}|\mathbf H|^2\bigl|\nabla^\perp\nu\bigr|^2=\frac{5n-8}{n+2}\left(f+|B_I|^2+|\mathring B_\nu|^2\right)\bigl|\nabla^\perp\nu\bigr|^2 .
    \end{equation*}
    Multiplying by $u$ and using $uf=|B_I|^2$ yields \eqref{eq:gradient_2_5_6}, which is \cite[(B.2)]{LeeNaffZhu_Planarity} for the pinching constant $\frac{4}{3n}$.
\end{proof}

Besides the vector field $b$, we set $\lambda_k\doteq\sum_{i,j}\mathring h^{n+1}_{ij}E_{ijk}$ and $\lambda\doteq\sum_k\lambda_ke_k$. By Remark \ref{rem:Q_term_alternative_expression},
\begin{equation}\label{eq:gradient_traceless_B_nu}
    \frac{1}{2}\nabla_k|\mathring B_\nu|^2=\lambda_k+b_k.
\end{equation}

\begin{lem}\label{lem:log_omega_expansion}
    At every point where $|\mathbf H|>0$ and $|B_I|>0$,
    \begin{equation}\label{eq:log_omega_expansion_5_6}
        (1-\theta) \Delta_\rho\log\omega
        =\frac{2}{|B_I|^2}\left[\theta u \mathcal R_f-\mathcal R_I\right]+\frac{2}{|B_I|^2} \mathcal G_\zeta,
    \end{equation}
    where
    \begin{equation}\label{eq:G_zeta_definition}
        \begin{aligned}
            \mathcal G_\zeta&\doteq\ \bigl|\nabla^\perp B_I\bigr|^2+\theta u \mathcal G_f-2\!\sum_{i,j,k}\!Q_{ijk}C_{ijk}\\
            &\qquad+\frac{|B_I|^2}{2}\left[\theta(1-\theta)\bigl|\nabla\log f\bigr|^2-2\theta\bigl\langle\nabla\log f,\zeta\bigr\rangle-|\zeta|^2\right].
        \end{aligned}
    \end{equation}
\end{lem}
\begin{proof}
    Since $(1-\theta)\log\omega=\log|B_I|^2-\theta\log f$ and $\Delta_\rho\log \varphi=\frac{\Delta_\rho \varphi}{\varphi}-|\nabla\log \varphi|^2$ for $\varphi>0$,
    \begin{equation*}
        (1-\theta)\Delta_\rho\log\omega=\frac{1}{|B_I|^2}\left[\Delta_\rho|B_I|^2-\theta u\Delta_\rho f\right]-\bigl|\nabla\log|B_I|^2\bigr|^2+\theta\bigl|\nabla\log f\bigr|^2 .
    \end{equation*}
    Expanding the bracket through \eqref{eq:interpolated_expansion}, and writing $\nabla\log|B_I|^2=\zeta+\theta\nabla\log f$ in the last two terms, gives \eqref{eq:log_omega_expansion_5_6}.
\end{proof}

\begin{lem}\label{lem:gradient_estimate_5_6}
    Let $\underline u>0$. Then, at every point where $|\mathbf H|>0$ and $u\geq \underline u$,
    \begin{equation}\label{eq:gradient_estimate_5_6}
        \mathcal G_\zeta\geq-\widetilde C |B_I|^2|\zeta|^2,
        \qquad
        \widetilde C\doteq\frac{1}{2(1-\theta)}\left(1+\frac{(1+\overline u)^2}{2\chi(1-\theta)\underline u^3}\right),
    \end{equation}
    where $\chi$ is defined in \eqref{eq:algebraic_constants_5_6}.
\end{lem}
\begin{proof}
    Since $u\geq \underline u>0$, one has $|B_I|>0$, and Cauchy--Schwarz in \eqref{eq:C_E_definition} gives $\bigl|\nabla^\perp\nu\bigr|^2\geq\frac{|C|^2}{|B_I|^2}$. Inserting \eqref{eq:gradient_1_5_6}, $\theta$ times \eqref{eq:gradient_2_5_6} and \eqref{eq:Q_C_contraction} into \eqref{eq:G_zeta_definition}, and using this bound, which is legitimate because the coefficients of $\bigl|\nabla^\perp\nu\bigr|^2$ in \eqref{eq:gradient_1_5_6} and \eqref{eq:gradient_2_5_6} are nonnegative, together with $\frac{f+|B_I|^2}{|B_I|^2}=\frac{1+u}{u}$, we get
    \begin{equation}\label{eq:G_zeta_first_bound}
        \begin{aligned}
            \mathcal G_\zeta&\geq\ |C|^2\left(\tau \frac{1+u}{u}+(\tau-1) \frac{|\mathring B_\nu|^2}{|B_I|^2}
            +\chi\left[1+u+u\frac{|\mathring B_\nu|^2}{|B_I|^2}\right]-1\right)\\
            &\qquad+\frac{1}{\eta} u |E|^2-2\langle E,C\rangle+2\bigl\langle\nabla\log|\mathbf H|,b\bigr\rangle\\
            &\qquad+\frac{|B_I|^2}{2}\theta(1-\theta)\bigl|\nabla\log f\bigr|^2-\theta|B_I|^2\bigl\langle\nabla\log f,\zeta\bigr\rangle-\frac{|B_I|^2}{2}|\zeta|^2,
        \end{aligned}
    \end{equation}
    with $\chi,\tau,\eta$ and (later) $\mu$ as in \eqref{eq:algebraic_constants_5_6}.

    For every index $k$, decompose $E_k\doteq(E_{ijk})_{i,j}$ and $C_k\doteq(C_{ijk})_{i,j}$, viewed as symmetric $2$-tensors, into their components parallel and orthogonal to $\mathring h^{n+1}$. At points where $|\mathring B_\nu|>0$ we set
    \begin{equation*}
        E_k=a_k \frac{\mathring h^{n+1}}{|\mathring B_\nu|}+E_k^\perp,
        \qquad
        C_k=c_k \frac{\mathring h^{n+1}}{|\mathring B_\nu|}+C_k^\perp,
    \end{equation*}
    while at points where $|\mathring B_\nu|=0$ we set $a_k=c_k=0$ and $E^\perp_k=E_k$, $C^\perp_k=C_k$. We write $a\doteq\sum_ka_ke_k$, $c\doteq\sum_kc_ke_k$, $|E^\perp|^2\doteq\sum_k|E^\perp_k|^2$, and similarly $|C^\perp|^2$ and $\langle E^\perp,C^\perp\rangle$. In both cases, $\lambda=|\mathring B_\nu| a$, $b=|\mathring B_\nu| c$, and
    \begin{equation*}
        |E|^2=|a|^2+|E^\perp|^2,\qquad |C|^2=|c|^2+|C^\perp|^2,\qquad \langle E,C\rangle=\langle a,c\rangle+\langle E^\perp,C^\perp\rangle .
    \end{equation*}
    Accordingly, the right-hand side of \eqref{eq:G_zeta_first_bound} splits as $\mathcal G_\zeta^\perp+\mathcal G_\zeta^\parallel$, where $\mathcal G_\zeta^\perp$ collects the terms containing $E^\perp$ and $C^\perp$, and $\mathcal G_\zeta^\parallel$ all the remaining ones.

    Since $\tau>1$ and $\chi>0$, the coefficient of $|C|^2$ in \eqref{eq:G_zeta_first_bound} is at least $\frac{\tau}{u}$. Hence
    \begin{equation*}
        \mathcal G_\zeta^\perp\geq\frac{1}{\eta}u|E^\perp|^2-2\langle E^\perp,C^\perp\rangle+\frac{\tau}{u}|C^\perp|^2
        =\frac{1}{\eta}u\left|E^\perp-\frac{\eta}{u}C^\perp\right|^2
        +\frac{\tau-\eta}{u}|C^\perp|^2\geq0,
    \end{equation*}
    because $\tau-\eta=\tau\frac{\chi-2}{\chi}\geq0$ by \eqref{eq:condition_C3_new}.

    Hence we focus on the parallel part $\mathcal G_\zeta^\parallel$.
    Differentiating \eqref{eq:algebraic_identity} and using the definition of $\zeta$ and \eqref{eq:gradient_traceless_B_nu},
    \begin{equation*}
        \frac{2}{3n}|\mathbf H|^2\nabla\log|\mathbf H|=(1+\theta u)\nabla f+2\lambda+2b+|B_I|^2\zeta.
    \end{equation*}
    Taking the inner product with $b=|\mathring B_\nu|c$ yields
    \begin{equation*}
        2\bigl\langle\nabla\log|\mathbf H|,b\bigr\rangle
        =\frac{3n|\mathring B_\nu|}{|\mathbf H|^2}\left[(1+\theta u)\langle\nabla f,c\rangle+|B_I|^2\langle\zeta,c\rangle\right]
        +\frac{6n|\mathring B_\nu|^2}{|\mathbf H|^2}\left(\langle a,c\rangle+|c|^2\right).
    \end{equation*}
    Substituted back into \eqref{eq:G_zeta_first_bound}, this allows two completions of squares. First,
    \begin{equation*}
        \begin{aligned}
            \frac{|B_I|^2}{2}\theta(1-\theta)&\bigl|\nabla\log f\bigr|^2-\theta|B_I|^2\bigl\langle\nabla\log f,\zeta\bigr\rangle
            +(1+\theta u)\frac{3n|\mathring B_\nu|}{|\mathbf H|^2}\langle\nabla f,c\rangle\\
            &\geq-\frac{\theta}{2(1-\theta)}|B_I|^2|\zeta|^2
            -\mu \frac{f^2(1+\theta u)^2}{|B_I|^2}\frac{9n^2|\mathring B_\nu|^2}{|\mathbf H|^4}|c|^2
            +\frac{f(1+\theta u)}{1-\theta}\frac{3n|\mathring B_\nu|}{|\mathbf H|^2}\langle\zeta,c\rangle .
        \end{aligned}
    \end{equation*}
    Second, since $\frac{3n|\mathring B_\nu|^2}{|\mathbf H|^2}-1=-\frac{3nf(1+u)}{|\mathbf H|^2}$ by \eqref{eq:algebraic_identity},
    \begin{equation*}
        \frac{1}{\eta}u|a|^2+2\left(\frac{3n|\mathring B_\nu|^2}{|\mathbf H|^2}-1\right)\langle a,c\rangle
        \geq-\eta \frac{9n^2f^2(1+u)^2}{|\mathbf H|^4 u} |c|^2 .
    \end{equation*}
    Inserting both into \eqref{eq:G_zeta_first_bound} yields
    \begin{align*}
        \mathcal G_\zeta^\parallel&\geq|c|^2\left(\tau\frac{1+u}{u}+(\tau-1)\frac{|\mathring B_\nu|^2}{|B_I|^2}+\chi\left[1+u+u\frac{|\mathring B_\nu|^2}{|B_I|^2}\right]+\frac{6n|\mathring B_\nu|^2}{|\mathbf H|^2}-1\right.
        \\&\left.\qquad\qquad-\mu\frac{f^2(1+\theta u)^2}{|B_I|^2}\frac{9n^2|\mathring B_\nu|^2}{|\mathbf H|^4}-\eta \frac{9n^2f^2(1+u)^2}{|\mathbf H|^4u}\right)
        \\&\qquad +\Lambda \langle\zeta, c\rangle -\frac{|B_I|^2}{2(1-\theta)}|\zeta|^2,
    \end{align*}
    where
    \begin{equation*}
        \Lambda\doteq |B_I|^2\frac{3n|\mathring{B}_\nu|}{|\mathbf{H}|^2} \frac{1+u}{u(1-\theta)}.
    \end{equation*}
    Set $s\doteq\frac{3nf}{|\mathbf H|^2}$ and $t\doteq\frac{3n|\mathring B_\nu|^2}{|\mathbf H|^2}$. Then $s>0$ and $s(1+u)+t=1$ by \eqref{eq:algebraic_identity}, so $t\in[0,1)$. Moreover, $\frac{|\mathring B_\nu|^2}{|B_I|^2}=\frac{t}{su}$ and $\frac{6n|\mathring B_\nu|^2}{|\mathbf H|^2}=2t$, and the last two terms in the coefficient of $|c|^2$ equal $-\mu\frac{st(1+\theta u)^2}{u}-\eta\frac{s^2(1+u)^2}{u}$. Therefore Lemma \ref{lem:algebraic_lemma_5_6} and Remark \ref{rem:conditions_C123_verified} bound the coefficient of $|c|^2$ from below by $\chi u$, and
    \begin{equation*}
        \begin{aligned}
            \mathcal G_\zeta^\parallel&\geq \chi u|c|^2+\Lambda \langle\zeta, c\rangle -\frac{|B_I|^2}{2(1-\theta)}|\zeta|^2
            \\&=\left|\sqrt{\chi u}\:c+\frac{\Lambda}{2\sqrt{\chi u}}\zeta\right|^2-\frac{\Lambda^2}{4\chi u}|\zeta|^2 -\frac{|B_I|^2}{2(1-\theta)}|\zeta|^2
            \\&\geq-\widetilde C|B_I|^2|\zeta|^2 ,
        \end{aligned}
    \end{equation*}
    where in the last inequality we used $\frac{\Lambda^2}{4\chi u}\leq\frac{(1+\overline u)^2}{4\chi(1-\theta)^2\underline u^3}|B_I|^2$, which follows from $|B_I|^2\leq\frac{1}{3n}|\mathbf H|^2$, $|\mathring B_\nu|^2\leq\frac{1}{3n}|\mathbf H|^2$ and $\underline u\leq u\leq \overline u$.
    Since $\mathcal G_\zeta=\mathcal G_\zeta^\perp+\mathcal G_\zeta^\parallel$ and
    $\mathcal G_\zeta^\perp\geq0$, this is \eqref{eq:gradient_estimate_5_6}.
\end{proof}

In particular, by \eqref{eq:log_omega_expansion_5_6}, the second inequality in \eqref{eq:reaction_inequality_n_geq_7} and \eqref{eq:gradient_estimate_5_6}, on $\{\omega>0\}\cap\{u\geq \underline u\}$ one has
\begin{equation}\label{eq:inequality_drift_laplacian_omega_complete}
    \Delta_\rho\log\omega\geq\frac{7}{1-\theta}|B_I|^2-2(1-\theta)\widetilde C\bigl|\nabla\log\omega\bigr|^2,
\end{equation}
which is the inequality announced in Section \ref{sec:translators}, with $\kappa_1=\frac{7}{1-\theta}=42$ and $\kappa_2=2(1-\theta)\widetilde C$.

\subsection{Proof of the codimension reduction}

\begin{thm}\label{thm:codim_reduction_n_geq_5}
    Let $n\geq5$, $p\geq1$, and let $x:M^n\to\mathbb R^{n+p}$ be a complete, connected translator without boundary satisfying \eqref{eq:main_pinching} for some $\varepsilon_0>0$. Then either $B\equiv0$, or $|\mathbf H|>0$ everywhere. In the latter case $B_I\equiv0$, $\nabla^\perp\nu\equiv0$, and $x(M)$ is contained in the affine subspace $A\doteq x(q_0)+W\subset\mathbb{R}^{n+p}$, for some $q_0\in M$ and some $(n+1)$-dimensional linear subspace $W\subseteq \mathbb R^{n+p}$. Moreover, $\vec v\in W$, the normal space of $x$ in $A$ is $\mathbb R\nu$, and $x:M\to A$ is an orientable codimension $1$ translator with the same velocity $\vec v$.
\end{thm}
\begin{proof}
    By Proposition \ref{prop:rho_laplacian_f}, either $f\equiv0$, in which case $B\equiv0$, or $f>0$ and $|\mathbf H|>0$ everywhere. Assume the latter.

    Since $f\leq\frac{1}{3n}|\mathbf H|^2\leq\frac{1}{3n}$ by \eqref{eq:f_upper_bound} and $|\mathbf H|\leq1$, and since $u\leq \overline u$ by \eqref{eq:uniform_pinching_bounds_u},
    \begin{equation*}
        0\leq\omega\leq\frac{1}{3n}(\overline u)^{\frac{1}{1-\theta}}<+\infty .
    \end{equation*}
    Assume, by contradiction, that $\omega^*\doteq\sup_M\omega>0$. By Proposition \ref{prop:omori_yau_translators} there is a sequence $\{q_j\}\subset M$ with
    \begin{equation*}
        \omega(q_j)>\omega^*-\frac1j,\qquad \bigl|\nabla\omega\bigr|(q_j)<\frac1j,\qquad \Delta_\rho\omega(q_j)<\frac1j .
    \end{equation*}
    For $j$ large enough, $\omega(q_j)\geq\frac{\omega^*}{2}$. Hence, since $\omega=fu^{\frac{1}{1-\theta}}$ with $u\leq\overline u$ and $f\leq\frac{1}{3n}$,
    \begin{equation*}
        f(q_j)\geq f_*\doteq\frac{\omega^*}{2(\overline u)^{\frac{1}{1-\theta}}}>0,
        \qquad
        u(q_j)\geq \underline u\doteq\left(\frac{3n\,\omega^*}{2}\right)^{1-\theta}>0.
    \end{equation*}
    Thus $|B_I|^2(q_j)=u(q_j)f(q_j)\geq \underline uf_*>0$, $\log\omega$ is defined near $q_j$, and
    \begin{equation*}
        |\zeta|(q_j)=(1-\theta)\frac{\bigl|\nabla\omega\bigr|(q_j)}{\omega(q_j)}\longrightarrow0,
        \qquad
        \limsup_j\Delta_\rho\log\omega(q_j)\leq\limsup_j\frac{\Delta_\rho\omega(q_j)}{\omega(q_j)}\leq 0 .
    \end{equation*}
    By \eqref{eq:inequality_drift_laplacian_omega_complete},
    \begin{equation*}
        (1-\theta)\Delta_\rho\log\omega(q_j)\geq7|B_I|^2(q_j)-2\widetilde C|\zeta|^2(q_j)\geq7\underline uf_*-2\widetilde C|\zeta|^2(q_j).
    \end{equation*}
    Taking the $\limsup_j$ of both sides gives $0\geq7\underline uf_*>0$, a contradiction. Hence $\sup_M\omega=0$ and, since $\omega\geq0$, $\omega\equiv0$. As $f>0$, this forces $B_I\equiv0$.

    By \eqref{eq:gradient_H_estimate}, $\nabla^\perp\nu\equiv0$. Since $B=h^{n+1}\otimes\nu$ and $\operatorname{tr}h^{n+1}=|\mathbf H|>0$, at every $q\in M$ the first normal space $N_1(q)$, that is, the orthogonal complement in $N_qM$ of $\{\xi\in N_qM\:|\:S_\xi=0\}$, equals $\mathbb R\nu(q)$. Hence $N_1$ has constant rank and is the parallel line bundle $\mathbb R\nu$, and the codimension reduction $x(M)\subseteq A$ follows from \cite{Erbacher_CodimensionReduction}.

    For every $q\in M$, $dx(T_qM)\subseteq W$ and $B(X,Y)\in W$ for all $X,Y\in T_qM$, therefore $\mathbf H=\operatorname{tr}_gB\in W$, $\nu=\mathbf H/|\mathbf H|\in W$ and $dx(T_qM)\oplus\mathbb R\nu(q)= W$. So $x:M\to A$ is two-sided with unit normal $\nu$ and scalar second fundamental form $h^{n+1}$. In particular, $M$ is orientable because it is two-sided and $A$ is orientable. Decomposing $\vec v=\vec v^\top+\mathbf H$, both summands lie in $W$, hence $\vec v\in W$. Finally, \eqref{eq:translator} still holds for $x:M\to A$ with the same velocity $\vec v$.
\end{proof}

\section{Compactness criterion and codimension reduction in low dimensions}\label{sec:codim_reduction_low}

\subsection{Extension of the compactness criterion}

We first extend \cite[Proposition 4.1]{LeeNaffZhu_Planarity} in dimension $n=2$, relying on White's theorem \cite{White} and on the total curvature formula of M\"uller and \v{S}ver\'ak \cite{MullerSverak_SurfacesFiniteTotalCurvature}.

\begin{prop}[Compactness criterion]\label{prop:translator_flatness_dim_234_intro}
    Let $n\geq 2$, $p\geq 1$ and let $x:M^n\to\mathbb R^{n+p}$ be a complete, connected immersion
    without boundary satisfying the pinching condition
    \begin{equation*}
        |B|^2 \leq \left(\frac{1}{n-1} - \varepsilon_0\right)|\mathbf{H}|^2,
    \end{equation*}
    for some constant $\varepsilon_0 > 0$. Then $M$ is either flat or compact. In the first case, $B\equiv 0$ and $x$ is an isometry onto an $n$-dimensional plane.
\end{prop}

\begin{rem}\label{rem:no_compact_translators}
    If, additionally, $M$ is a translator, Proposition \ref{prop:translator_flatness_dim_234_intro} implies that $x$ is an isometry onto an $n$-dimensional plane, since closed translators do not exist in the Euclidean space. In fact, for a closed translator, integrating the identity $\Delta \langle x, \vec v \rangle = \langle \mathbf{H}, \vec v \rangle = |\mathbf{H}|^2$ over $M$ forces $\mathbf{H} \equiv 0$, which would make $M$ a closed minimal submanifold of $\mathbb{R}^{n+p}$. But then the integral of $\Delta|x|^2=2n+2\langle x,\mathbf H\rangle=2n$ over $M$ implies $0=2n\operatorname{Vol}(M)$, a contradiction.
\end{rem}

\begin{proof}[Proof of Proposition \ref{prop:translator_flatness_dim_234_intro}]
    If $n\geq 3$ the result is a consequence of \cite[Proposition 4.1]{LeeNaffZhu_Planarity}, a generalization of a compactness criterion in \cite{LynchNguyen_PinchedAncientSolutionsHighCodimMCF} through a recent work of Lee--Topping \cite{LeeTopping_ManifoldsPIC1Curvature}. Moreover, by \cite[Lemma 3.2]{LynchNguyen_PinchedAncientSolutionsHighCodimMCF},
    \begin{equation*}
        \mathcal R\geq \frac{\varepsilon_0}{2}|\mathbf H|^2\cdot \operatorname{Id}\geq 0,
    \end{equation*}
    where $\mathcal R$ is the curvature operator of the induced metric $g$ on $M$ and $\operatorname{Id}$ is the identity operator on the bundle of antisymmetric bi-vectors. Hence flatness implies $|\mathbf H|\equiv 0$ and, consequently, $B\equiv 0$ from the pinching condition. To conclude, $x$ is an isometry onto an $n$-dimensional plane thanks to \cite[Theorem 1]{ONeill_IsometricImmersion}.

    It remains to prove the $n=2$ case. If $M$ is compact there is nothing to prove, so suppose, first, that $M$ is orientable and non-compact. In the rest of the proof, $K$ denotes the Gaussian curvature of $M$, whereas $K^+$ and $K^-$ denote its positive and negative parts, respectively. The Gauss equation, together with the pinching hypothesis, gives
    \begin{equation*}
        2K = |\mathbf H|^2 - |B|^2 \ge \varepsilon_0|\mathbf H|^2 \ge 0,
    \end{equation*}
    so $K \ge 0$ everywhere. It follows that $K \ge \frac{\varepsilon_0}{2}|\mathbf H|^2 = \frac{\varepsilon_0}{2}(2K + |B|^2)$, so
    \begin{equation}\label{eq:pinching_estimate_dimension_2}
        |B|^2 \le \frac{2(1-\varepsilon_0)}{\varepsilon_0}K,
    \end{equation}
    an estimate we will use repeatedly. Since $K \ge 0$, certainly $\int_M K^- = 0 < \infty$, so \cite[Theorem 1]{White} applies and yields two facts: $\int_M K^+ < \infty$, hence $\int_M K < \infty$, and $M$ has finite topological type, i.e.\ $M \cong S \setminus \{q_1, \dots, q_k\}$ for some connected, compact surface $S$. Non-compactness of $M$ forces $k \ge 1$.

    It remains to show $K\equiv0$, which via \eqref{eq:pinching_estimate_dimension_2} gives $B\equiv0$ and finishes the proof as in the case $n\geq3$. Integrating \eqref{eq:pinching_estimate_dimension_2} and using $\int_MK<\infty$ gives $\int_M|B|^2<\infty$. Hence \cite[Corollary 4.2.5]{MullerSverak_SurfacesFiniteTotalCurvature} applies to the complete, connected, non-compact, oriented surface $M$ and gives
    \begin{equation*}
        \int_M K = 2\pi\left(\chi(M)-\sum_{i=1}^k m_i\right),
    \end{equation*}
    where $m_i\geq1$ is the multiplicity of the end $a_i$ in the sense of \cite[Definition 4.2.4]{MullerSverak_SurfacesFiniteTotalCurvature}. Since $M$ is connected and non-compact, $H_2(M;\mathbb Z)=0$, so that $\chi(M)=1-b_1(M)\leq1$. Therefore
    \begin{equation*}
        0\leq\int_MK\leq2\pi(1-k)\leq0,
    \end{equation*}
    so $\int_MK=0$ and, since $K\geq0$ is continuous, $K\equiv0$.

    Lastly, suppose that the surface $M$ is non-orientable and non-compact. Consider the orientable double cover $\pi:\widetilde M\to M$, which is a local isometry when $\widetilde M$ is endowed with the pullback metric.
    Then $\widetilde M$ satisfies several properties. First, since $\pi$ is a Riemannian covering map and $M$ is complete, so is $\widetilde M$. Second, $\widetilde M$ is orientable (by construction) and connected, because $M$ is non-orientable. Third, $\widetilde M$ is immersed in the Euclidean space $\mathbb R^{2+p}$ through the map $\tilde x\doteq x\circ \pi$, and its second fundamental form $\widetilde B$ is the pullback of $B$ through $\pi$, i.e. $\widetilde B=\pi^*B$. Hence $|\widetilde B|^2=|B|^2\circ \pi$, $|\widetilde{\mathbf H}|^2=|\mathbf H|^2\circ\pi$ and the pinching holds for $\widetilde M$ with the same constant $\frac{1}{n-1}=1$ and the same $\varepsilon_0>0$. Since $\pi$ is continuous and surjective and $M$ is non-compact, $\widetilde M$ is non-compact as well, and the previous argument gives $\widetilde B\equiv0$. As $\widetilde B=\pi^*B$ and $\pi$ is surjective, $B\equiv0$, and $x$ is an isometry onto an affine plane as in the case $n\geq3$. In particular, $M$ is orientable, so this case does not actually occur.
\end{proof}

\subsection{Codimension reduction in dimensions \texorpdfstring{$2\leq n\leq 4$}{2<=n<=4}}

\begin{cor}\label{cor:codim_reduction_n_leq_4}
    Let $2\leq n\leq4$, $p\geq1$, and let $x:M^n\to\mathbb R^{n+p}$ be a complete, connected translator without boundary satisfying \eqref{eq:main_pinching} for some $\varepsilon_0>0$. Then $B\equiv0$ and $x$ is an isometry onto an affine $n$-plane containing the direction $\vec v$.
\end{cor}
\begin{proof}
    Since $\frac{4}{3n}\leq\frac{1}{n-1}$ for $n\leq4$, the hypothesis \eqref{eq:main_pinching} with the same $\varepsilon_0$ allows us to apply Proposition \ref{prop:translator_flatness_dim_234_intro}. By Remark \ref{rem:no_compact_translators} the compact alternative does not occur, hence $B\equiv0$ and $x$ is an isometry onto an affine $n$-plane. Lastly, $\vec v^\perp=\mathbf H\equiv0$, so the plane contains the direction $\vec v$.
\end{proof}

\section{Proof of the main results}\label{sec:proof_main_theorems}

The two ranges of dimensions have been treated separately in Sections \ref{sec:codim_reduction_n_geq_5} and \ref{sec:codim_reduction_low}. Putting them together proves Theorem \ref{thm:main_codim_reduction}.

\begin{proof}[Proof of Theorem \ref{thm:main_codim_reduction}]
    If $2\leq n\leq4$, Corollary \ref{cor:codim_reduction_n_leq_4} gives alternative $(1)$. If $n\geq5$, Theorem \ref{thm:codim_reduction_n_geq_5} gives alternative $(2)$ unless $B\equiv0$. In the latter case $x$ is a complete totally geodesic immersion, hence an isometry onto an affine $n$-plane, as shown in \cite[Theorem 1]{ONeill_IsometricImmersion}, and the plane contains $\vec v$ because $\vec v^\perp=\mathbf H\equiv0$.
\end{proof}

\begin{proof}[Proof of Theorem \ref{thm:rigidity_result}]
    Since $c_n\leq\frac{4}{3n}$, condition \eqref{eq:main_pinching} holds with the same $\varepsilon_0$, so Theorem \ref{thm:main_codim_reduction} and, for $n\in\{2,3,4\}$, Corollary \ref{cor:codim_reduction_n_leq_4} apply. The latter proves the result when $n\in\{2,3,4\}$, so suppose $n\geq 5$.

    If $B\equiv 0$, alternative $(1)$ of Theorem \ref{thm:main_codim_reduction} shows that $x$ is an isometry onto an $n$-dimensional plane containing $\vec v$. Otherwise, by alternative $(2)$, $|\mathbf H|>0$ everywhere and $x:M\to A\cong \mathbb R^{n+1}$ is an orientable codimension $1$ translator with the same velocity $\vec v$; Moreover, it also satisfies the same pinching. Since the scalar mean curvature $H=\operatorname{tr}_g h^{n+1}=|\mathbf H|>0$ is everywhere positive and $c_n\leq\frac{1}{n-2}$, \cite[Lemma 5.1]{HuiskenSinestrari_MCFConvexAncientSolutions} implies that $M$ is uniformly $2$-convex. Consequently, \cite[Theorem 1.1]{SpruckSun_ConvexityTranslatingSolitons} applies to show that $M$ is locally strictly convex, which ensures that the second fundamental form is positive definite, and hence that $M$ is uniformly $(n-1)$-convex. Sacksteder's results \cite{Sacksteder_HypersurfacesNoNegSectionalCurvatures,Sacksteder_RigidityHypersurfaces} (see also \cite[Theorem 3.1]{Lima_Survey}) give that $x$ is an embedding and $x(M)$ is the smooth boundary $\partial\Omega$ of a convex body $\Omega\subseteq A$. So, \cite[Theorem 1.1]{BourniLangfordLynch_CollapsingAndNoncollapsingAncientMCF} applies with $\Omega_t=\Omega+t\vec v$, yielding that the inscribed radius of $\Omega_t$ is at least $H^{-1}$ at each point of $\partial\Omega_t$. Moreover, the exterior ball radius is $+\infty$ because $\Omega_t$ is convex. Thus $x$ is $1$-noncollapsed in the sense of \cite{Haslhofer_BowlSoliton} and, being complete, embedded, orientable and uniformly $2$-convex, it is an embedding onto the rotationally symmetric bowl soliton with velocity $\vec v$ by \cite[Theorem 1.2]{Haslhofer_BowlSoliton}.
\end{proof}

\appendix

\section{An algebraic inequality}\label{sec:algebraic_lemma_5_6}

We establish, in this appendix, the algebraic inequality needed in the proof of Lemma \ref{lem:gradient_estimate_5_6}.

\begin{lem}\label{lem:algebraic_lemma_5_6}
    Let $n\geq 5$ and let $\theta\in(0,1)$. Set \begin{equation}\label{eq:algebraic_constants_5_6}
        \chi \doteq \theta \frac{5n-8}{n+2}, \quad
        \tau \doteq \frac{3(n-1)}{n+2}, \quad
        \eta \doteq \frac{2\tau}{\chi}=\frac{6(n-1)}{\theta(5n-8)}, \quad
        \mu \doteq \frac{1}{2\theta(1-\theta)}.
    \end{equation}
    Suppose
    \begin{align}
        2\eta+\tau&\geq\mu+1 , \tag{C1}\label{eq:condition_C1_new}
        \\ \tau+1=\frac{4n-1}{n+2}&\geq\frac{\theta}{2(1-\theta)} ,
        \tag{C2}\label{eq:condition_C2_new}
        \\ \chi&\geq2. \tag{C3}\label{eq:condition_C3_new}
    \end{align}
    Then, for all $s>0$, $t\in[0,1)$ and $u>0$ that satisfy the constraint $s(1+u)+t=1$,
    \begin{equation*}
        \tau \frac{1+u}{u}+(\tau-1) \frac{t}{su}+\chi\left(1+u+\frac ts\right)+2t-1
        -\mu \frac{s t (1+\theta u)^2}{u}-\eta \frac{s^2(1+u)^2}{u}\ \geq\ \chi u .
    \end{equation*}
\end{lem}
\begin{proof}
    Substituting $s=\frac{1-t}{1+u}$ and multiplying by $u>0$, the claim is equivalent to $\Phi(t,u)\geq0$ on $[0,1)\times(0,+\infty)$, where
    \begin{equation}\label{eq:Phi_definition_and_split}
        \Phi(t,u)\doteq\Phi_0(t)+u\left[\tau-1+2t+\frac{(\tau-1) t}{1-t}+\frac{\chi(1+tu)}{1-t}\right]
        -\mu t(1-t)\left[\frac{(1+\theta u)^2}{1+u}-1\right],
    \end{equation}
    and $\Phi_0(t)\doteq \tau-\eta+\eta t(2-t)+\frac{(\tau-1) t}{1-t}-\mu t(1-t)$.

    First, since $t\geq 0$, $1-t>0$ and $\tau-1\geq 0$, $\frac{(\tau-1) t}{1-t}\geq(\tau-1) t$ and
    \begin{equation*}
        \Phi_0(t)-(\tau-\eta)\geq t\left[\eta(2-t)+(\tau-1)-\mu(1-t)\right]\doteq t \ell(t).
    \end{equation*}
    Because $\ell$ is affine with $\ell(0)=2\eta+\tau-\mu-1\geq 0$ by \eqref{eq:condition_C1_new} and $\ell(1)=\eta+\tau-1>0$, $\Phi_0(t)\geq \tau -\eta$.

    Second, since $2\theta\leq1+\theta^2$ one has $(1+\theta u)^2\leq(1+u)(1+\theta^2u)$, so that $\frac{(1+\theta u)^2}{1+u}-1\leq\theta^2u$, and $\mu\theta^2=\frac{\theta}{2(1-\theta)}$. Therefore, since $\frac{(\tau-1) t}{1-t}\geq(\tau-1) t$ and $\frac{\chi(1+tu)}{1-t}\geq\chi$, \eqref{eq:Phi_definition_and_split} implies that
    \begin{equation*}
        \Phi(t,u)\geq\Phi_0(t)+u\left[\tau+\chi-1
        +t\left(1+\tau-\frac{\theta}{2(1-\theta)}(1-t)\right)\right]\geq\Phi_0(t),
    \end{equation*}
    where the last inequality uses $\tau+\chi>1$, which is easily implied by \eqref{eq:condition_C3_new}, together with \eqref{eq:condition_C2_new}.

    Finally,
    \begin{equation*}
        \Phi(t,u)\geq\Phi_0(t)\geq\tau-\eta=\tau \frac{\chi-2}{\chi}\geq0,
    \end{equation*}
    thanks to \eqref{eq:condition_C3_new}.
\end{proof}

\begin{rem}\label{rem:conditions_C123_verified}
    For $\theta=\frac56$ one has $\mu=\frac{18}{5}$, $\frac{\theta}{2(1-\theta)}=\frac52$ and $\chi=\frac{5(5n-8)}{6(n+2)}$. Hence
    \begin{equation*}
        2\eta+\tau-\mu-1=\frac{32n^2-169n+344}{5(n+2)(5n-8)},
    \end{equation*}
    whose numerator has negative discriminant, so \eqref{eq:condition_C1_new} holds for every $n\geq2$. Moreover, \eqref{eq:condition_C2_new} is equivalent to $n\geq4$, and \eqref{eq:condition_C3_new} to $13n\geq64$, that is, $n\geq5$.
\end{rem}

\section*{Data availability}
Data sharing not applicable to this article as no datasets were generated or analyzed during
the current study.

\section*{Conflict of interest statement}
On behalf of all authors, the corresponding author states that there is no conflict of interest.

\bibliographystyle{amsalpha}
\bibliography{bibliography}

\end{document}